\documentclass [11pt,oneside,a4paper,mathscr]{amsart}
\usepackage{amsfonts, amsmath, amssymb, amsthm,mathtools, stmaryrd}
\usepackage{verbatim}
\usepackage[normalem]{ulem}
\usepackage{tikz-cd}
\usepackage{enumitem}

\usepackage{array}

\usepackage{hyperref}

\theoremstyle{plain}
\newtheorem{thm}{Theorem}[section]
\newtheorem{cor}[thm]{Corollary}
\newtheorem{lemma}[thm]{Lemma}

\newtheorem{prop}[thm]{Proposition}

\numberwithin{equation}{section}

\theoremstyle{definition}
\newtheorem{defn}[thm]{Definition}
\newtheorem{example}[thm]{Example}
\newtheorem{rmk}[thm]{Remark}

\theoremstyle{remark}

\newcommand{\BC}{{\mathbb{C}}}

\newcommand{\BZ}{{\mathbb{Z}}}

\newcommand{\CA}{{\mathcal A}}
\newcommand{\CB}{{\mathcal B}}
\newcommand{\CC}{{\mathcal C}}
\newcommand{\CD}{{\mathcal D}}
\newcommand{\CE}{{\mathcal E}}
\newcommand{\CF}{{\mathcal F}}
\newcommand{\CG}{{\mathcal G}}

\newcommand{\CL}{{\mathcal L}}

\newcommand{\CO}{{\mathcal O}}

\newcommand{\Fc}{{\mathfrak{c}}}

\newcommand{\ch}{{\mathrm{ch}}}
\newcommand{\td}{{\mathrm{td}}}

\DeclareFontFamily{OT1}{rsfs}{}
\DeclareFontShape{OT1}{rsfs}{n}{it}{<-> rsfs10}{}
\DeclareMathAlphabet{\curly}{OT1}{rsfs}{n}{it}

\newcommand\Hom{\operatorname{Hom}}

\newcommand{\Aut}{\operatorname{Aut}}

\newcommand\Spec{\operatorname{Spec}}

\newcommand{\Coh}{\mathrm{Coh}}

\newcommand{\Pic}{\mathop{\rm Pic}\nolimits}

\newcommand{\id}{\mathrm{id}}
\newcommand{\Tr}{\mathrm{Tr}}

\newcommand{\TTr}{{\mathbb{T}r}}

\newcommand{\FM}{\mathsf{FM}}

\newcommand{\Cats}{\mathfrak{Cats}}
\newcommand{\GCats}{G\textup{-}\mathfrak{Cats}}

\newcommand{\Ind}{\mathrm{Ind}}
\newcommand{\Func}{\mathrm{Func}}

\newcommand{\HH}{\mathrm{HH}}

\begin{document}
%\baselineskip=14.5pt
%\title{Notes on equivariant categories}
\title{On equivariant derived categories}

\author{Thorsten Beckmann}
\address{Universit\"at Bonn, Mathematisches Institut}
\email{beckmann@math.uni-bonn.de}

\author{Georg Oberdieck}
\address{Universit\"at Bonn, Mathematisches Institut}
\email{georgo@math.uni-bonn.de}
\date{\today}

\begin{abstract}
We study the equivariant category associated to a finite group action on the derived category of coherent sheaves of a smooth projective variety. We discuss decompositions of the equivariant category and faithful actions, prove the existence of a Serre functor, give a criterion for the equivariant category to be Calabi--Yau, and describe an obstruction for a subgroup of the group of auto-equivalences to act. As application we show that the equivariant category of any symplectic action on the derived category of an elliptic curve is equivalent to the derived category of an elliptic curve.
\end{abstract}

\maketitle

\setcounter{tocdepth}{1} 
\tableofcontents

\section{Introduction} 
Equivariant categories appear naturally when
expressing sheaves on a quotient space $X/G$ 
in terms of sheaves on the space $X$ on which a given group $G$ acts.
For example, given a finite group $G$ acting freely on a complex quasi-projective variety $X$
the category of coherent sheaves on the quotient variety $X/G$ is
the $G$-equivariant category of coherent sheaves on $X$:
\[ \Coh(X)_G = \Coh(X/G). \]

More generally, if we start with a category of sheaves on a space
and an (abstract) group action on the category,
the equivariant category
may be viewed
as a form of 'non-commutative quotient'.
Since the action does not have to come from an action on the underlying space,
in general this quotient exists only on a categorical level,
i.e.\ it is not clear whether it is again the category of sheaves on some space.
If we nevertheless hope for some geometric meaning of this category,
it is natural to explore what geometric structures it possesses.
The focus of this paper is to discuss several of these structures (decompositions, duality, cohomology)
for finite group actions on the derived category of coherent sheaves on a smooth projective variety.

We give a short overview of the content of the paper:
In Section~\ref{sec:grp actions} we define group actions on $\BC$-linear additive categories.
We give a natural obstruction class in $H^3(G, \BC^{\ast})$
that governs whether a subgroup of the group of auto-equivalences
can define an action on the category.
We also discuss criteria for the vanishing of this class,
and give an example that the obstruction is effective.

In Section~\ref{sec:equivant cats} we define equivariant categories, recall and prove their $2$-categorical universal properties,
show (following work of Romagny \cite{R}) that the equivariant categories can be taken successively,
and discuss the action of the dual group.

In Section~\ref{sec:decomp equiv} we study
under what conditions the equivariant category is indecomposable.
If it is not, we show that each summand is the equivariant category of a related group action.
Our approach relies on $G$-fixed stability conditions.
These induce stability conditions on the equivariant category \cite{MMS}
and their existence makes equivariant categories behave quite similar to an actual geometric space,
at least with regard to taking components.

In Section~\ref{subsec:Serrefunctor} we prove that
given a triangulated category with a Serre functor
the equivariant category also carries a Serre functor.

In Section~\ref{sec:cohomology} we discuss properties of the Hochschild cohomology of equivariant categories.
By relying on work of Perry \cite{Perry} we give a criterion for the equivariant category
of a Calabi--Yau category to be again Calabi--Yau.
We also describe the induced action by the group of characters on Hochschild cohomology.

In Section~\ref{sec:equiv_cat_ell_curves} we illustrate our methods by
determining the equivariant categories of an elliptic curve
with respect to Calabi--Yau group actions.

We attempted to make the presentation of this paper as self-contained as possible,
so that it can serve also as an introduction to the subject.
Many of the topics we discuss here are scattered around in the literature,
and we expect many others, where we have not found any references,
to be folklore and known to the experts.
Useful sources for the equivariant category of derived categories
are the works of Elagin \cite{Elagin},
Shinder \cite{Shinder}, Perry \cite{Perry},
Kuznetsov and Perry \cite{KP} and others.
We refer to these paper for related discussions and applications.

Studying equivariant categories of derived categories of smooth projective varieties
becomes a rich subject only if interesting group actions are known beyond
geometric automorphism. The motivation for us arose from
symplectic group actions on the derived category of K3 and abelian surfaces.
In this case string theory and holomorphic symplectic geometry alike provide a wide range of interesting examples.
In many of these cases, the equivariant categories are in a non-trivial way equivalent to
the derived category of another symplectic surface.
We refer to \cite{BeckmannOberdieckModuli} for more details
and applications to fixed loci of holomorphic symplectic varieties.

\subsection{Conventions}
We always work over $\BC$. All categories are assumed to be small. 

\subsection{Acknowledgements}
We thank Daniel Huybrechts and Johannes Schmitt for useful conversations
on Theorem~\ref{lem:obstruction to action} and the example in Section~\ref{subsec:does not act} respectively.

\section{Group actions on categories}\label{sec:grp actions}
Let $G$ be a finite group and let
$\CD$ be a category.
\subsection{Definition} \label{sec:defn grp actions}
An action $(\rho, \theta)$ of $G$ on $\CD$ consists of
\begin{itemize}
\item for every $g \in G$ an auto-equivalence $\rho_g \colon \CD \to \CD$,
\item for every pair $g,h \in G$ an isomorphism of functors $\theta_{g,h} \colon \rho_{g} \circ \rho_h \to \rho_{gh}$
\end{itemize}
such that for all triples $g,h,k \in G$ we have the commutative diagram
\begin{equation} \label{associativity}
\begin{tikzcd}[row sep=large, column sep = large]
\rho_{g} \rho_{h} \rho_{k} \ar{r}{\rho_g  \theta_{h,k}} \ar{d}{\theta_{g,h} \rho_k} & \rho_{g} \rho_{hk} \ar{d}{\theta_{g,hk}} \\
\rho_{gh} \rho_{k} \ar{r}{\theta_{gh,k}} & \rho_{ghk}.
\end{tikzcd}
\end{equation}
We will often write $g$ for $\rho_g$.

Recall the $2$-category $\Cats$ of categories,
where the objects are categories,
the morphisms are functors between categories
and the $2$-morphisms are natural transformations.
Similarly we have the 2-category $\GCats$ of categories with a $G$-action.
A morphism or \emph{$G$-functor}
\[ (f,\sigma) \colon (\CD, \rho, \theta) \to (\CD', \rho', \theta') \]
between categories with $G$-actions
is a pair of a functor $f \colon \CD \to \CD'$ together with 2-isomorphisms
$\sigma_g \colon f \circ \rho_g \to \rho_g' \circ f$
such that $(f,\sigma)$ intertwines the associativity relations on both sides, i.e.\ such that the following diagram commutes:
\[
\begin{tikzcd}
f \rho_g \rho_h \ar{r}{f \theta_{g,h}} \ar{d}{\sigma_g \rho_h} & f \rho_{gh} \ar{dd}{\sigma_{gh}} \\
\rho_{g}' f \rho_h \ar{d}{\rho_g \sigma_h} & \\
 \rho_g' \rho_h' f \ar{r}{\theta_{g,h}' f} & \rho_{gh}' f.
\end{tikzcd}
\]
%By \cite[Lem.\ 3.5]{Shinder} every $G$-functor induces a functor on equivariant categories $\tilde{f} \colon \CD_G \to \CD'_G$ such that $p \circ \tilde{f} = f \circ p$.
A $2$-morphism of $G$-functors $(f,\sigma) \to (\tilde{f}, \tilde{\sigma})$ is a $2$-morphism $t\colon f \to \tilde{f}$ that inter-twines the $\sigma_g$,
i.e.\ such that the following diagram commutes:
\[
\begin{tikzcd}
f \circ \rho_g \ar{r}{\sigma_g} \ar{d}{t \rho_g} & \rho'_g \circ f \ar{d}{\rho_g' t} \\
\tilde{f} \circ \rho_g \ar{r}{\tilde{\sigma}_g} & \rho'_g \circ \tilde{f}.
\end{tikzcd}
\]

An action of $G$ on $\CD$ is \emph{strict} if $\theta_{g,h} = \id$ for all $g,h \in G$.
In particular, this implies that $\rho_1=\id$.
By \cite[Thm.\ 5.4]{Shinder} every $G$-action on $\CD$ is equivalent to a strict $G$-action on some equivalent category $\CD'$.
Here we say that a $G$-action $(\rho, \theta)$ on $\CD$ is \textit{equivalent} to a $G$-action $(\rho', \theta')$
on $\CD'$ if
we have an equivalence in $\GCats$, i.e.\ 
$(\CD, \rho, \theta) \cong (\CD', \rho', \theta')$.
Because of this,
by passing to an equivalent category one can (and we often will) assume that the action is strict.

\subsection{Obstruction to actions}
Let $\Aut \CD$ be the group of  \emph{isomorphism classes} of auto-equivalences of $\CD$. %$\CD \to \CD$.
Hence two equivalences 
\[ f_1, f_2 \colon \CD \to \CD \]
are identified in $\Aut \CD$ if and only if there exists an isomorphism of functors $f_1 \xrightarrow{\cong} f_2$.
Every group action on $\CD$ yields a subgroup of $\Aut \CD$.
%by sending a group element to the class of the auto-equivalence it defines.
For $\BC$-linear categories the converse however does not always hold and is obstructed by a class in the group cohomology of $G$ as explained in the following theorem. 

To state it we will need a different notion of equivalence for group actions.
We say that $G$-actions $(\rho,\theta)$ and $(\rho',\theta')$ on $\CD$ are \emph{isomorphic}
if there exists a $G$-functor of the form $(\id_{\CD},\sigma) \colon (\CD,\rho,\theta) \to (\CD,\rho',\theta')$.

\begin{thm} \label{lem:obstruction to action}
Let $\CD$ be a $\BC$-linear category and assume that $\Hom(\id_{\CD}, \id_{\CD}) = \BC \id$. Let $G \subset \Aut \CD$ be a finite subgroup.
\begin{enumerate}
\item[(a)] There exists a class in $H^3(G, \BC^{\ast})$ which vanishes if and only if there exists an action of $G$ on $\CD$
whose image in $\Aut \CD$ is $G$.
Moreover, the set of isomorphism classes of such actions is a torsor under $H^2(G,\BC^{\ast})$.
\item[(b)] There exits a finite group $G'$ and a surjection $G' \twoheadrightarrow G$ such that
$G'$ acts on $\CD$ and the induced map $G' \to \Aut \CD$ factors over the quotient map to $G$.
\item[(c)] If $G = \BZ_n$, then we can take $\BZ_{n^2} \to \BZ_n$ in (b).
\end{enumerate}
\end{thm}

In the group cohomology $H^{i}(G, \BC^{\ast})$ above
the group $G$ acts trivially on the coefficient group $\BC^{\ast}$.
For cyclic groups one has
\[ H^i(\BZ_n, \BC^{\ast}) =
\begin{cases}
\BC^{\ast} & \text{ if } i=0, \\
\BZ_n & \text{ if } i \text{ odd, } \\
0 & \text{ if } i>0 \text{ even. }
\end{cases}
\]
Hence the obstruction in part (a) of Theorem~\ref{lem:obstruction to action} can be non-trivial even for cyclic groups. We refer to Section~\ref{subsec:does not act} for a class of such examples in a geometrically well-behaved situations.

\begin{proof}
(a) For every $g \in G$ choose a functor $\rho_g \colon \CD \to \CD$ with image $g$ in $\Aut \CD$
and for every pair $g,h \in G$ choose isomorphisms $\theta_{g,h} \colon \rho_g \rho_h \to \rho_{gh}$.
Then for every triple $g,h,k \in G$ consider the composition given by applying the maps in \eqref{associativity} once counterclockwise around the square,
\begin{equation} \label{eq:sample}
(g)(h)(k) \xrightarrow{\theta_{g,h}k}
(gh)(k) \xrightarrow{\theta_{gh,k}}
(ghk) \xrightarrow{\theta_{g, hk}^{-1}} 
(g)(hk) \xrightarrow{g \theta_{h,k}^{-1}}
(g)(h)(k)
\end{equation}
where we have written $(g)$ for $\rho_g$.
The assumption $\Hom(\id_{\CD}, \id_{\CD}) = \BC$ yields that
this composition is a scalar multiple of the identity,
which we call $c(g,h,k)$.
The first step of the proof is to prove that the assignment
\[ c \colon G^3 \to \BC^{\ast} \]
is a cocycle. Since the $G$-action on $\BC^{\ast}$ is taken to be trivial here,
this boilds down to showing that for all quadruples $g,h,k,l\in G$ we have
\begin{equation} \label{to prove}
c(g,h,kl) c(gh,k,l) = c(h,k,l) c(g, hk, l) c(g,h,k).
\end{equation}

We start with the left hand side.
The constant $c(g,h,kl)$ times the identity is the composition
\begin{equation} \label{a}
(g)(h)(kl) \to (gh)(kl) \to (ghkl) \to (g)(hkl) \to (g)(h)(kl)
\end{equation}
where the maps are all given by the corresponding $\theta$'s as in \eqref{eq:sample}.\footnote{
It is very suggestive to write this composition vertically:
%We can write this composition also more suggestively as
\begin{equation*} 
\begin{tikzcd}[row sep=tiny, column sep = tiny]
(g)\ (h)\ (kl) \ar{d} \\
(gh) \ (kl) \ar{d} \\
(ghkl) \ar{d} \\
(g) \ (hkl) \ar{d} \\
(g) \ (h) \ (kl).
\end{tikzcd}
\end{equation*}
We invite the reader to re-write the other maps below in a similar form,
in order to make the various compositions more clear.
For brevity we will stick to the horizontal notation.}
Similarly to above, $c(gh,k,l) \id$ is the composition
\begin{equation} \label{b}
(gh)(k)(l) \to (ghk)(l) \to (ghkl) \to (gh)(kl) \to (gh)(k)(l).
\end{equation}
Composing \eqref{a} with the $2$-morphism $(g)(h)(kl) \to (gh)(kl)$
and precomsing it with its inverse yields
\begin{equation} \label{aa}
c(g,h, kl) \id = \left[ (gh)(kl) \to (ghkl) \to (g)(hkl) \to (g)(h)(kl) \to (gh)(kl) \right].
\end{equation}
Composing \eqref{b} by $(gh)(k)(l) \to (gh)(kl)$ and
precomposing by its inverse yields
\begin{equation} \label{bb}
c(gh,k,l) \id = \left[ (gh)(kl) \to (gh)(k)(l) \to (ghk)(l) \to (ghkl) \to (gh)(kl) \right].
\end{equation}
By considering the composition \eqref{aa} $\circ$ \eqref{bb} and noting that the
last map in \eqref{bb} is precisely the inverse of the first map in \eqref{aa},
we hence find that
$c(g,h,kl) c(gh,k,l)$ times the identity is given by the following composition
\begin{equation} \label{c}
(gh)(kl) 
\to (gh)(k)(l)
\to (ghk)(l)
\to (ghkl)
\to (g)(hkl)
\to (g)(h)(kl)
\to (gh)(kl).
\end{equation}

We now turn to the right hand side of \eqref{to prove}.
Arguing in a similar manner shows that
$c(h,k,l) c(g,hk, l) c(g,h,k)$ times the identity
is equal to the composition
\begin{multline} \label{d}
(g) (h) (k) (l) \to (gh) (k) (l) \to (ghk) (l) \to (ghkl) \\
\to (g) (hkl)
\to (g) (h) (kl)
\to (g) (h) (k) (l).
\end{multline}
We see that except for the two respective outer arrows, the compositions \eqref{c}
and \eqref{d} agree.
Hence to prove the desired equation \eqref{to prove} it remains to prove that
the compositions
\begin{gather*}
(g) (h) (kl) \to (gh) (kl) \to (gh)(k)(l) \\
(g)(h) (kl) \to (g)(h)(k)(l) \to (gh)(k)(l)
\end{gather*}
%\[ (g) (h) (kl) \to (gh) (kl) \to (gh)(k)(l),
%\quad (g)(h) (kl) \to (g)(h)(k)(l) \to (gh)(k)(l) \]
agree, or equivalently, that we have a commutative diagram
\[
\begin{tikzcd}[row sep=large, column sep = large]
(g)(h)(kl) \ar{r}{\theta_{gh}(kl)} \ar{d}{ (g)(h) \theta_{kl}^{-1} } 
& (gh)(kl) \ar{d}{ (gh) \theta_{k,l}^{-1} } \\
(g)(h)(k)(l) \ar{r}{\theta_{g,h} (k) (l)} & (gh)(k)(l).
\end{tikzcd}
\]
This follows from the following.

\begin{lemma}
Let $f,f',g, g' \colon \CD \to \CD$ be functors and let $\alpha \colon f \to f'$ and $\beta \colon g \to g'$ be
natural transformations. Then we have a commutative diagram
\[
\begin{tikzcd}
fg \ar{d}{f \beta} \ar{r}{\alpha g} & f' g \ar{d}{f' \beta} \\
f g' \ar{r}{\alpha g'} & f' g'
\end{tikzcd}
\]
\end{lemma}
\begin{proof}
For every $A \in \CD$ we need to prove that
\[
\begin{tikzcd}
fg A \ar{d}{f \beta^A} \ar{r}{\alpha^{gA}} & f' g \ar{d}{f' \beta^A} \\
f g' \ar{r}{\alpha^{g'A}} & f' g'
\end{tikzcd}
\]
commutes. This follows when we apply the condition that $\alpha$ is a natural transformation
to the morphism $\beta^A \colon g A \to g' A$.
\end{proof}

Therefore, the above defines a cocycle $c \colon G^3 \to \BC^{\ast}$.
It dependeds on the choice of representative $\rho_g$ and the choice of $\theta_{g,h}$.
By the assumption $\Hom(\id_{\CD}, \id_{\CD}) = \BC$ for any second choice of
isomorphism $\theta'_{g,h} \colon (g)(h) \to (gh)$
we have $\theta'_{g,h} = \lambda(g,h) \theta_{g,h}$ for some $\lambda(g,h) \in \BC^{\ast}$.
The new cocycle $c'$ one obtains in this way is
\[ c'(g,h,k) = \lambda(g,h) \lambda(gh,k) \lambda(g,hk)^{-1} \lambda(h,k)^{-1} c(g,h,k) \]
and hence $c$ differs from $c'$ by a coboundary.
Similarly, any other choice of representative $\rho'_g$ changes $c$ by at most a coboundary.
Hence we obtain a well-defined class
\[ c \in H^3(G,\BC^{\ast}) \]
depending only on $G \subset \Aut \CD$. 
By construction, it vanishes if and only if there is a choice of $\theta_{g,h}$ which satisfies \eqref{associativity}, hence if and only if there is an action $(\rho, \theta)$ of $G$ on the category $\CD$.
Moreover, once an action $(\rho,\theta)$ has been found,
we obtain any other action $\theta'$ by multiplying $\theta$ with an arbitrary 
$\lambda \colon G^2 \to \BC^{\ast}$ which is a $2$-cocycle. 
One checks that such a $\lambda$ is a coboundary if and only if the actions $\theta$ and $\theta'$ are isomorphic. This proves part (a).

(b) We will show that given $\alpha \in H^{3}(G, \BC^{\ast})$
there exists a surjection $G' \twoheadrightarrow G$ from a finite group $G'$ such that the restriction of $\alpha$ to $H^3(G', \BC^{\ast})$ vanishes.
Since $H^{3}(G,\BC^{\ast})$ is finite, $\alpha$ is the image of
some $\beta \in H^3(G, \BZ_n)$ for some $n$ under the map induced by the inclusion $\BZ_n \to \BC^{\ast}$.
Recall that we have the commutative diagram
\[
\begin{tikzcd}
H^3(G', \BZ_n) \ar{r} & H^3(G', \BC^{\ast}) \\
H^3(G, \BZ_n) \ar{u} \ar{r} & H^3(G, \BC^{\ast}) \ar{u}.
\end{tikzcd}
\]
We find that it is enough to construct a surjection $G' \twoheadrightarrow G$  such that the image of $\beta$ in $H^3(G', \BZ_n)$ vanishes.

It is %well-
known that $\beta$ corresponds to a crossed module
\[ 1 \to \BZ_n \to N \to E \to G \to 1 \]
where the action of $G$ on $\BZ_n$ is trivial \cite[IV.5]{Brown}.
The pullback of $\beta$ along $E \to G$ corresponds to the crossed module
\[ 1 \to \BZ_n \to N \to E \times_G E \to E \to 1. \]
Since this has a section $s\colon E \to E \times_G E$ and $G$ acts trivially on $\BZ_n$, it is equivalent to the trivial crossed module. Concretely, there is a morphism
\[
\begin{tikzcd}
1 \ar{r}&\BZ_n \ar{r}{\id}\ar{d}{\id}& \BZ_n \ar{r}{0}\ar{d}& E \ar{r}{\id}\ar{d}{s}& E\ar{r}\ar{d}{\id}& 1\\
1 \ar{r} & \BZ_n \ar{r} & N \ar{r} & E\times_G E \ar{r} & E \ar{r} & 1
\end{tikzcd}
\]
of crossed modules, which by definition is a map of long exact sequences compatible with the actions of the groups in the crossed modules. The upper crossed module corresponds to the trivial crossed module.

Therefore the pullback of the class $\beta$ to $E$ vanishes.
Moreover, by \cite[p.\ 502]{Ellis} we can choose $E$ to be finite, so setting $G' = E$ yields the claim.\\
(c) Consider the short exact sequence
\[
0 \to \BZ_n \to \BZ_{n^2} \to \BZ_n \to 0
\]
and apply the Lyndon--Hochschild--Serre spectral sequence. We have 
\[ E_2^{3,0}=H^3(\BZ_n,\BC^\ast) \cong \BZ_n, \]
and the composite map
\[
E_2^{3,0} \twoheadrightarrow E_{\infty}^{3,0} \subset H^3(\BZ_{n^2},\BC^\ast) 
\]
corresponds to the pullback map
\[ H^3(\BZ_n, \BC^{\ast}) = \BZ_{n} \to H^3(\BZ_{n^2}, \BC^{\ast}) = \BZ_{n^2}. \]
Since $H^2(\BZ_{n^2},\BC^\ast)=0$, the differential $d \colon E_2^{1,1}\to E_2^{3,0}$ is injective. Hence, the assertion follows from $E_2^{1,1}=H^1(\BZ_n,H^1(\BZ_n,\BC^\ast))\cong \BZ_n$ and $E_2^{3,0} \cong \BZ_n$. %Alternatively, the non-trivial cross modules corresponding to non-zero classes in $H^3(\BZ_n,\BC^\ast)$ are given by the crossed module
%\[
%0 \to \BZ_n \to \BZ_{n^2} \to \BZ_{n^2} \to \BZ_n \to 0. 
%\]
%Arguing as in (b) yields the claim. 
\end{proof}

\subsection{Action via Fourier--Mukai transforms}
Let $X$ be a smooth complex projective variety and let
\[ D^b(X) = D^b (\Coh(X))  \]
be the bounded derived category of coherent sheaves on $X$.
Given an object
\[ \CE \in D^b(X \times X) \]
the Fourier--Mukai transform $\FM_{\CE} \colon D^b(X) \to D^b(X)$
with kernel $\CE$ is defined as
\[ \FM_{\CE}(A) = q_{\ast}( p^{\ast}(A) \otimes \CE) \]
where $p,q \colon X \times X \to X$ are the projections
and all functors are derived.
For a multitude of applications it is very useful to know whether
a given endofunctor of $D^b(X)$ is given by a Fourier--Mukai transform.
For fully faithful exact functors this was answered affirmatively by Orlov \cite{OrlovsThm}.

\begin{defn}
A \emph{Fourier--Mukai action} of $G$ on $D^b(X)$ consists of\footnote{We write $\CE \circ \CF$ to indicate the composition of correspondences $\CE,\CF$.}
\begin{itemize}
\item for every $g \in G$ a Fourier--Mukai kernel $\CE_g \in D^b(X \times X)$,
\item for every pair $g,h \in G$ an isomorphism $\theta_{g,h} \colon \CE_g \circ \CE_h \to \CE_{gh}$
\end{itemize}
such that for all $g,h,k$ the diagram \eqref{associativity} commutes with $\rho_g$ replaced by $\CE_g$.
\end{defn}

By associating to a kernel its Fourier--Mukai transform
we see that any Fourier--Mukai action on $D^b(X)$ 
induces a group action on $D^b(X)$ in the sense of Section~\ref{sec:defn grp actions}.
We have the following converse.

\begin{lemma} \label{lemma:FMaction}
Let $X$ be smooth complex projective variety and let $G$ be a finite group.
Then any $G$-action on $D^b(X)$ is induced by a unique Fourier--Mukai action.
\end{lemma}
\begin{proof}
Given $(\rho, \theta)$, by Orlov's theorem \cite{OrlovsThm} for every $g \in G$ there exists a kernel $\CE_g \in D^b(X \times X)$ with $\FM_{\CE_g} \cong \rho_g$.
By uniqueness of the kernels there also exists isomorphisms $\theta'_{g,h} \colon \CE_g \circ \CE_h \cong \CE_{gh}$.
Since $\Hom(\CE_g, \CE_g) = \BC$, arguing as in the proof of Theorem~\ref{lem:obstruction to action}
yields a class $\alpha$ in $H^3(G, \BC^{\ast})$ which vanishes
if and only if after replacing $\theta'$ by a boundary the pair $(\CE_g, \theta')$ defines a Fourier--Mukai action.
By passing to Fourier--Mukai transforms one sees that $\alpha$ is the same class
as the obstruction class defined by $G \subset \Aut D^b(X)$ and hence has to vanish (since $G$ acts on $D^b(X)$).
This shows that there is a Fourier--Mukai action of $G$ on $D^b(X)$ such that $\FM_{\CE_g} \cong \rho_g$.

A similar argument shows further that the possible isomorphism classes of such Fourier--Mukai $G$-actions
are a torsor under $H^2(G,\BC^{\ast})$.
Moreover the map that associated to a Fourier--Mukai action the induced action on $D^b(X)$
is equivariant with respect to the action of $H^2(G, \BC^{\ast})$.
%and that these classes match the corresponding classes for the $G$-action.
This implies that we can also match (in a unique way, up to isomorphism) the $2$-isomorphisms $\theta$ for the action on $D^b(X)$ and for the Fourier--Mukai action.
\end{proof}

\section{Equivariant categories} \label{sec:equivant cats}
Let $(\rho, \theta)$ be an action of a finite group $G$
on an additive $\BC$-linear category $\CD$.

\subsection{Definition}
The equivariant category $\CD_{G}$ is defined as follows:
\begin{itemize}
\item Objects of $\CD_G$ are pairs $(E,\phi)$ where $E$ is an object in $\CD$ and $\phi = (\phi_g \colon E \to \rho_{g} E)_{g \in G}$ is a family of isomorphisms
such that
\begin{equation} 
\label{compatibility}
\begin{tikzcd}
E \ar[bend right]{rrr}{\phi_{gh}} \ar{r}{\phi_g} & \rho_{g}E \ar{r}{\rho_g \phi_h} &  \rho_{g} \rho_{h} E \ar{r}{\theta_{g,h}^E} & \rho_{gh} E
\end{tikzcd}
\end{equation} 
commutes for all $g,h \in G$.
\item A morphism from $(E,\phi)$ to $(E', \phi')$ is a morphism $f \colon E \to E'$ in $\CD$
which commutes with linearizations, i.e.\ such that
\[
\begin{tikzcd}
E \ar{r}{f} \ar{d}{\phi_g} & E' \ar{d}{\phi'_g} \\
g E \ar{r}{\rho_gf} & g E'
\end{tikzcd}
\]
commutes for every $g \in G$.
\end{itemize}

The definition of morphism can be reformulated as follows.
For any objects $(E,\phi)$ and $(E',\phi')$ in $\CD_G$
consider the action of $G$ on $\Hom_{\CD}(E,E')$ via 
%$f \colon E \to E'$ to $(\phi'_g)^{-1} \circ \rho_g(f) \circ \phi_g$.
\[ f \mapsto (\phi'_g)^{-1} \circ \rho_g(f) \circ \phi_g. \]
Then we have
\[ \Hom_{\CD_G}( (E,\phi), (E',\phi') ) = \Hom_{\CD}(E,E')^G. \]

\subsection{Induction and restriction functor}
\label{subsec:indresfunc}
Given a subgroup $H \subset G$ we have a \emph{restriction functor}
\[ \mathrm{Res}^G_H \colon \CD_G \to \CD_H \]
defined by restricting the linearization of an equivariant object to the subgroup $H$.

In the opposite direction we have an \emph{induction functor}
\[ \Ind_H^{G} \colon \CD_H \to \CD_G \]
which is constructed as follows: Let $g_i$ be representatives of the cosets $G/H$,
where one of the $g_i$ equals $1 \in G$ (representing the unit coset).
Then we set
\[ \Ind_H^{G}(E,\phi) = \left( \bigoplus_{i} \rho_{g_i} E, \phi^G \right), \]
where for every $g \in G$ the restriction of $\phi^G_g$ to the summand $\rho_{g_j} E$ is defined by
\[ \rho_{g_j} E \xrightarrow{\rho_{g_j} \phi_{h}} \rho_{g_j} \rho_{h} E \xrightarrow{\theta_{g_j,h}}
\rho_{g_j h} \xrightarrow{\theta_{g, g_i}^{-1}} \rho_{g} \rho_{g_i} E
%\hookrightarrow \rho_g \oplus_i \rho_{g_i} E
\]
where $g_i$ and $h$ are defined by $g g_i = g_j h$. % for some (necessarily unique) $h \in H$.
By a similar argument as in \cite[Lem.\ 3.8]{Elagin} one has that
$\Ind_H^{G}$ is both left and right adjoint to $\mathrm{Res}^G_H$, see also \cite[Lem.\ 3.3]{Perry}. 

In the case of the trivial subgroup, $H=1$, the restriction
and induction functors specialize to the \emph{forgetful functor}
\[ p \colon \CD_G \to \CD, \quad (E,\psi) \mapsto E \]
which forgets the linearization, and the \emph{linearization functor}
\begin{equation} 
q\colon \CD \to \CD_G, \quad E \mapsto \left( \oplus_{g \in G} \rho_g E, \phi \right). \label{linearlization_functor} \end{equation}
%where $\phi$ is the canonical linearization. 

\subsection{Universal property}
Equivariant categories can be viewed as limits in the category $\Cats$.
To explain this let us view a finite group $G$
as the 2-category $\CG$ with one object,
1-morphisms given by the elements of $G$, and only identities as 2-morphisms.
Giving a $G$-action on a category $\CD$
is then equivalent to giving a $2$-functor
\[ \CG \to \Cats \]
which takes the unique object in $\CG$ to the category $\CD$.
The equivariant category is the $2$-limit of this morphism:
\[ \CD_G = 2\text{-lim}(\CG \to \Cats). \]
This yields a universal property of the equivariant category
that we state in explicit terms below.
The proof follows by general theory,
but for concreteness we will sketch a direct argument.
We also refer to \cite[App.\ A]{Wasch} for a discussion of
$2$\text{-(co)limits} in $\Cats$, and to
\cite[Prop.\ 4.4]{GK} for more details in the case of equivariant categories.

Given a category $\CA$, let $\iota(\CA)$ denote the category $\CA$ endowed with the trivial $G$-action. 
For any two categories $\CA,\CB$ let
\[ \Hom(\CA, \CB) \]
be the category whose objects are functors $\CA \to \CB$ and whose morphisms are natural transformations of such functors, and similarly for categories with $G$-action.

\begin{prop} \label{prop:UP1}
Let $G$ be a finite group which acts on a $\BC$-linear category $\CD$.
Then for every category $\CA$ we have a bifunctorial equivalence of categories
\[
\Hom_{\Cats}{(\CA, \CD_G)} \cong \Hom_{\GCats}{(\iota(\CA),\CD)}.
\]
\end{prop}
\vspace{5pt}

The proposition implies that every $G$-functor $\iota(\CA) \to \CD$ can be factored via the equivariant category:
\[
\begin{tikzcd}
& & \CD_G \ar{d}{p} \\
\iota(\CA) \ar{rr}{ G\text{--functor}} \ar[dotted]{urr}{\exists !} &&  \CD
\end{tikzcd}
\]
Conversely, the forgetful functor $p \colon \CD_G \to \CD$ carries a natural structure of a $G$-functor.
Hence composing any functor $\CA \to \CD_G$ with $p$ yields a $G$-functor $\iota(\CA) \to \CD$.
%The second case is similar with $q$ playing the role of $p$.

\begin{proof}[Proof of Proposition~\ref{prop:UP1}]
We can assume that the $G$-action on $\CD$ is strict.

Let $\CA$ be a category and let $(f, \sigma) \colon \iota(\CA) \to \CD$ be a $G$-functor.
By definition of a $G$-functor, the $2$-isomorphisms $\sigma_g \colon f \to \rho_g f$
fit into the commutative diagram
\begin{equation} \label{dsf1111}
\begin{tikzcd}
f \ar{r}{\sigma_g} \ar[bend right]{rr}{\sigma_{gh}} & \rho_g f \ar{r}{\rho_g \sigma_h} & \rho_g \rho_h f.
\end{tikzcd}
\end{equation}
Thus for any $E \in \CA$ the collection 
\[ (\sigma_g)^E \colon fE \to \rho_g fE \]
is a $G$-linearization of $fE$.
Moreover, since $\sigma_g$ is a natural transformation, for any morphism $\psi \colon E \to F$ in $\CA$ the map $f \psi \colon f E \to fF$ is $G$-invariant with respect to these linearizations.
This yields a functor
\[ F \colon \CA \to \CD_G, \ \ E \mapsto (fE, \sigma^E),\  \psi \mapsto f \psi. \]
One further checks that any
natural transformation of $G$-functors $(f,\sigma) \to (f', \sigma')$ yields
a natural transformation $F \to F'$ of the corresponding functors.
These assignments define a functor
\begin{equation} \label{hom1} \Hom_{\GCats}{(\iota(\CA),\CD)} \to  \Hom_{\Cats}{(\CA, \CD_G)}. \end{equation}

Conversely, the $2$-isomorphisms
\[ \tau_g \colon p \xrightarrow{\cong} \rho_g \circ p \]
defined using the linearization $\tau_g^{(E,\phi)} = \phi_g$ for all $g \in G$ give $p$ the structure
of a $G$-functor $(p,\tau)\colon \iota(\CD_G) \to \CD$. 
Any functor $F \colon \CA \to \CD_G$ is automatically a $G$-functor $(F,\id) \colon \iota(\CA)\to \iota(\CD_G)$ and hence we obtain
a $G$-functor 
\[ (p F, \tau F)\colon \iota(\CA) \to \CD. \]
Similarly, for any natural transformation $t \colon F \to F'$ of functors $\CA \to \CD_G$
we obtain the natural transformation of $G$-functors 
\[ pt \colon (pF, \tau F) \to (pF', \tau F'). \]
This yields an inverse to \eqref{hom1}.
\end{proof}

If $\CD$ is an abelian category we also have a universal property with respect to the functor $q$.
It says that $\CD_G$ is the 2-colimit in the 2-category of abelian categories where the morphisms are left exact functors. 
We state it here for completeness, but it will not be essential later on. %and can be skipped.
For any two abelian categories let us denote by
\[ \Hom_{\text{l.e.}}(\CA, \CB) \subset \Hom(\CA,\CB) \] 
the subcategory of left-exact additive functors.

\begin{prop} \label{prop:UP2}
Let $G$ be a finite group which acts on an abelian category $\CD$.
Then for any abelian category $\CA$ we have the equivalence of categories
\[
\Hom_{\Cats, \textup{l.e.}}{(\CD_G, \CA)} \cong \Hom_{\GCats,, \textup{l.e.}}{(\CD, \iota(\CA))}.
\]
\end{prop}
Similar results hold for dg-categories (or certain stable $\infty$-categories)
and will play a role in determining the Hochschild cohomology of the equivariant category,
see Section~\ref{sec:cohomology} below and \cite{Perry}.
\begin{proof}
We assume that the action is strict.
Since $q(gE) = qE$ for all objects $E$ in $\CD$,
the linearization functor $q \colon \CD \to \CD_G$ defines a $G$-functor $(q, \id) \colon \CD \to \iota(\CD_G)$.
We define a functor
\begin{equation} \Hom_{\Cats}{(\CD_G, \CA)} \to \Hom_{\GCats}{(\CD, \iota(\CA))} \label{func a} \end{equation}
by pre-composing with $q$, i.e.\ 
by sending a functor $F \colon \CD_G \to \CA$ to $(Fq, F \id)$, and a natural transformation
$t$ of such functors to $tq$.

Conversely, let $(f,\sigma) \colon \CD \to \iota(\CA)$ be a $G$-functor.
One checks immediately (or see e.g.\ \cite[Lem.\ 3.5]{Shinder}) that
$(f,\sigma)$ naturally lifts to a functor
$\tilde{f} \colon \CD_G \to \CA_G$
%such that $p \circ \tilde{f} = f \circ p$ and 
such that as a composition of $G$-functors we have
\begin{equation}
\label{fqqf}
(\tilde{f}, \id) \circ (q,\id)
= (q,\id) \circ (f,\sigma).
\end{equation}
Since $G$ acts trivially on $\CA$, 
objects in $\CA_G$ are pairs $(E,\phi)$ where $E$ is an object in $\CA$
endowed with a $G$-action
given by the linearization $\phi_g \colon E \to E$.
Since $\CA$ is abelian and hence has finite limits, 
we have a functor
\[ (- )^G \colon \CA_G \to \CA \]
which associates to $(E,\phi)$ its $G$-invariants $E^G$. It satisfies $( - )^G \circ q = \id_{\CA}$.
We set
\begin{equation} F = (-)^G \circ \tilde{f} \colon \CD_G \to \CA. \label{FFdef} \end{equation}

Similarly, any natural transformation $t \colon f \to f'$ lifts to a natural transformation
$\tilde{t} \colon \tilde{f} \to \tilde{f'}$. % see \cite[Lem.\ 3.8]{Shinder}.
Taking $G$-invariants we obtain a natural transformation
$(-)^G \tilde{t} \colon F \to F'$.
This yields a functor
\begin{equation} \Hom_{\GCats}{(\CD, \iota(\CA))} \to \Hom_{\Cats}{(\CD_G, \CA)}. \label{func b} \end{equation}

We need to show that \eqref{func a} and \eqref{func b} are quasi-inverse to each other
when restricted to the subcategory of left-exact functors.
Consider a $G$-functor $(f,\sigma) \in \Hom_{\GCats, \textup{l.e.}}{(\CD, \iota(\CA))}$
and let $F$ be defines as in \eqref{FFdef}. Then
$F$ is a left-exact additive functor and we have the composition of $G$-functors
\begin{align*}
(F, \id) \circ (q,\id) & = ( (-)^G, \id) \circ (\tilde{f}, \id) \circ (q,\id) \\
& \overset{\eqref{fqqf}}{=} ( (-)^G, \id) \circ (q,\id) \circ (f,\sigma) \\
& = (f,\sigma).
\end{align*}

Conversely, given $F \in \Hom_{\Cats, \textup{l.e.}}{(\CD_G, \CA)}$, define the $G$-functor
\[ (f,\sigma) = (F,\id) \circ (q,\id)\colon \CD \to \iota(\CA)  \]
which is left-exact and additive,
and consider its lift $\tilde{f} \colon \CD_G \to \CA_G$. We need to show that
\[ F \cong (-)^G \circ \tilde{f}. \]

Given $(B,\phi) \in \CD_G$, consider the object $qB = qp(B,\phi)= (\oplus_{g} gB, \phi_{\mathrm{can}})$.
The linearization $\phi_h \colon B \to hB$ yields the
morphism $q\phi_h \colon qB \to qB$ in $\CD_G$ given by
\begin{equation}\label{dddd}
\left( \bigoplus_{g \in G} gB, \phi_{\mathrm{can}} \right) \xrightarrow{ \left( g \phi_h \right)_{g \in G}} 
\left( \bigoplus_{g \in G} ghB, \phi_{\mathrm{can}} \right) 
= \left(\bigoplus_{g \in G} gB, \phi_{\mathrm{can}} \right).
\end{equation}
This defines a $G$-action on the object $qB$, i.e.\ a homomorphism $G \to \Hom_{\CD_G}(qB, qB)$.
The morphism
\[ \bigoplus_{g \in G} \phi_g \colon B \to \bigoplus_{g \in G} gB \]
is $G$-invariant with respect to the linearizations $\phi$ and $\phi_{\mathrm{can}}$,
and defines an isomorphism from $(B,\phi)$ with the $G$-invariants $(qB)^G$.

By definition the element $\tilde{f} (B,\phi)$ is the pair $(F qB, Fq\phi_h)$
where the linearization $Fq\phi_h$ is obtained as the composition
\[ FqB \xrightarrow{Fq \phi_h} Fqh B = FqB. \]
Hence it is precisely $F$ applied to the morphism \eqref{dddd}.
Since $F$ is left exact and hence commutes with finite limits,
we find
\[ \left( \tilde{f} (B,\phi) \right)^G = 
\left( F(qB, q\phi_h) \right)^G = % \overset{(*)}{=}
F  \left( \left( qB, q\phi_h \right)^G \right) \cong F (B,\phi). \] 
The final step (which is left to the reader) is to show that \eqref{func a} and \eqref{func b} are inverse to each other
on natural transformations.
\end{proof}

\subsection{Taking equivariant categories successively}
The following result shows that when determining equivariant categories
it is sufficient to consider simple groups $G$.
We also refer to \cite[Rem.\ 2.4]{R} for the parallel statement
for stacks.

\begin{prop} \label{rem:successive_quotients}
If $H \subset G$ is a normal subgroup, then there exists an induced action of $G$ on $\CD_H$
such that $(p,\id) \colon \CD_H \to \CD$ is a $G$-functor.
The action on $\CD_H$ is isomorphic to an action which factors through $G/H$ and we have $\CD_G \cong (\CD_H)_{G/H}$.
\end{prop}

\begin{proof}
We assume that the action is strict.
For every $g \in G$ define a functor 
$\bar{\rho}_g \colon \CD_H \to \CD_H$
by letting it act on morphisms by $\rho_g$ and on objects by
\[ \bar{\rho}_g \colon (A,\phi) \mapsto (\rho_g A, \phi')\ \text{ with }\ \phi'_h = \rho_g \phi_{g^{-1} h g}. \]
One checks that the assignment $g \mapsto \bar{\rho}_g$ defines a strict $G$-action on $\CD_H$
for which $(p,\id)$ is a $G$-functor.

For the second part we define a new action $\tilde{\rho}_g$.
Choose representatives $g_1, \dots, g_n \in G$ for the elements in $G/H$, where we take the identity element for the coset of the identity.
Given any element $g \in g_i H$ we set
\[ \tilde{\rho}_g=\bar{\rho}_{g_i}. \]
For any two elements $g \in g_i H$ and $g' \in g_j H$ write $g_i g_j = g_k h$ (here $g_k$ and $h$ only depend on $g_i, g_j$).
Then define
\[ \theta_{g,g'} \colon \tilde{\rho}_{g} \tilde{\rho}_{g'} = \bar{\rho}_{g_i} \bar{\rho}_{g_j} = \bar{\rho}_{g_k h} \to \bar{\rho}_{g_k} = \tilde{\rho}_{gg'} \]
by associating to $(A,\phi) \in \CD_H$ the morphism
$\bar{\rho}_{g_k h} (A,\phi) \xrightarrow{ \rho_{g_k} \phi_h^{-1} } \bar{\rho}_{g_k} (A,\phi)$.
The action of $(\tilde{\rho}, \theta)$ on $\CD_H$ is isomorphic to $(\bar{\rho},\id)$
and factors through a $G/H$-action.

The adjunction
\[
\Hom_{G/H\textup{-}\mathfrak{Cats}}( \iota_{G/H}(\CA), \CD_H ) \xrightarrow{\cong} \Hom_{\GCats}(\iota(\CA), \CD)
\]
follows by a direct check.
Using Proposition~\ref{prop:UP1} and the 2-categorical Yoneda lemma
yields the isomorphism $\CD_G \cong (\CD_H)_{G/H}$.
\end{proof}

\subsection{Action of the dual group}
The group of characters of $G$,
\[ G^{\vee} = \{ \chi \colon G \to \BC^{\ast} \mid \chi \text{ homomorphism} \}, \]
acts on the equivariant category $\CD_G$ via the identity on morphisms and by
\[ \chi \cdot (E,\phi) = (E, \chi \phi) \]
on objects, where we let $\chi \phi$ denote the linearization $(\chi\phi)_g=\chi(g) \phi_g \colon E \to \rho_{g} E$. The cocycle condition for $\chi \phi$ is satisfied because for any $g,h \in G$ one has
\[ \chi(gh) = \chi(g) \chi(h). \]

We discuss a typical example arising in geometry.

\begin{example} \label{ex:dual action}
Let $G$ be a finite group acting on an complex quasi-projective algebraic variety $X$.
The $G$-action induces an action on the category of
$\Coh(X)$ of coherent sheaves on $X$
by sending a sheaf $E$ to its pushforward $g_{\ast}E$ under the automorphism $g \colon X \to X$.
If $G$ acts freely, then we have the following well-known equivalence
of the equivariant category
with the category of coherent sheaves on the quotient variety $X/G$:\footnote{If the $G$-action
is not free, then parallel statements apply to the stack quotient $[X/G]$.}
\begin{equation} \label{equiv quotient} \Coh(X/G) \xrightarrow{\cong} \Coh(X)_G. \end{equation}
The equivalence is given by pullback of sheaves along the quotient map $\pi \colon X \to X/G$.
Conversely the linearization $\phi$ of some $(E,\phi)$ is the descent datum of the sheaf $E$ with respect to $\pi$.
Under the equivalence the structure sheaf of $\CO_{X/G}$
corresponds to the equivariant sheaf $(\CO_X,1)$, where we write
$1$ for the canonical linearization
\[ \CO_X \to g_{\ast} \CO_X \]
given by pullback of functions along $g$.
For every character $\chi \in G^{\vee}$ consider the line bundle
\[ \CL_{\chi} \in \Pic(X/G) \]
which corresponds to the twisted linearization $(\CO_X, \chi)$.
Then tensoring with $\CL_{\chi}$ on $\Coh(X/G)$
corresponds under the equivalence \eqref{equiv quotient} to the dual action of $\chi \in G^{\vee}$ on the equivariant category.
In particular, the line bundles $\CL_{\chi}$ are all torsion. \qed
\end{example}

If the group $G$ is abelian, then $G^{\vee}$ is called the \textit{dual group} and is non-canonically isomorphic to $G$.
In this case we have the following:
\begin{prop}[Reversion, {\cite[Thm.\ 1.3]{Elagin}}] \label{Reciprocity}
Let $\CD$ be an idempotent complete additive category over $\BC$ with an action of a finite abelian group $G$.
Then there is an equivalence
$(\CD_G)_{G^{\vee}} \cong \CD$.
\end{prop}
%\begin{proof}
%We sketch the construction of the equivalence.
%An object of $(\CD_G)_{G^{\vee}}$ is a triple $(F,\phi, \psi)$ consisting of $(F,\phi) \in \CD_G$ and morphisms
%$\psi_{\chi} \colon (E,\phi) \to (E,\chi \phi)$ for every $\chi \in G^{\vee}$ satisfying the cocycle condition.
%Given $E \in \CD$ one takes $(F,\phi) = q(E)$ and $\psi_{\chi}|_{gE} = \chi(g)^{-1} \id$.
%\begin{comment}
%Here we use that for every $h \in G$ one has the commutative diagram
%\[
%\begin{tikzcd}
%g E \ar{rrrr}{\psi_{\chi}|_{gE} = \chi(g)^{-1} \id } \ar{d}{\phi_h = \id} & & &  & g E \ar{d}{\chi(h) \phi_h = \chi(h) \id} \\
%h h^{-1}g E \ar{rrrr}{ h \psi_{\chi}|_{h^{-1} g E} = \chi(h^{-1}g)^{-1}} \id & & & & h h^{-1} g E
%\end{tikzcd}
%\]
%\end{comment}
%Conversely given $(F,\phi, \psi)$ one defines $E$ by taking the eigenspace decomposition of $\psi$.
%\end{proof}

The proposition can be applied for example to the categories $\Coh(X)$ and $D^b(X)$, since they are both idempotent complete. 

\begin{example}
\label{ex:actionPic0}
If $X$ is an algebraic variety and $H \subset \Pic(X)$ a finite subgroup
which acts on $\Coh(X)$ by tensoring, then the equivariant category $\Coh(X)_H$
is equivalent to $\Coh(\tilde{X})$ where $\tilde{X}$ is the cover
\[ \tilde{X} = \Spec\left( \bigoplus_{\CL \in H} \CL^{-1} \right), \]
see \cite[Thm.\ 7.5]{Elagin} for a discussion.

In particular, if $G$ is taken to be abelian in Example~\ref{ex:dual action},
and we take the subgroup of line bundles $\CL_{\chi}$ for all $\chi \in G^{\vee}$,
then one recovers
\[ \Coh(X/G)_{G^{\vee}} \cong \Coh(X). \]
Hence this is a basic example of the reversion principle of Proposition~\ref{Reciprocity}. \qed
\end{example}

A \emph{$G$-linearization} of an object $E \in \CD$ is an object $\tilde{E} \in \CD_G$ such that $p \tilde{E} \cong E$.
We say that an object $E$ is \emph{$G$-linearizable} if it admits a $G$-linearization.
Equivalently, it lies in the essential image of the functor $p$. 

By work of Ploog,
the action of the dual group yields the following useful description of the set of $G$-linearizations of a simple object.

\begin{lemma}(\cite[Lem.\ 1]{Ploog}) \label{rmk:Ploog_simple}
Let a finite group $G$ act on a $\BC$-linear category $\CD$ and consider a simple object $E$, i.e.\ $\Hom(E,E)=\BC$. Then there exists a class in $H^2(G,\BC^\ast)$ which vanishes if and only if $E$ admits a $G$-linearization. 
Furthermore, the dual group $G^{\vee}$ acts freely and transitively on the set
of (isomorphism classes) of linearizations.
\end{lemma}

\subsection{A subgroup of the group of auto-equivalences which does not act}
\label{subsec:does not act}
Let $\CD$ be a triangulated category and let
$\tau \colon \CD \to \CD$
be an auto-equivalence of order $4$ (so $\tau^k \not\cong \id$ unless $4|k$) which defines a strict $\BZ_4$-action on $\CD$.
Let $\langle \tau^2 \rangle \cong \BZ_2$ denote the subgroup generated by $\tau^2$ and let 
\[ \CD' = \CD_{\langle \tau^2 \rangle} \]
be the equivariant category.
As in Proposition~\ref{rem:successive_quotients}, $\tau$ induces an auto-equivalence $\bar{\tau} \colon \CD' \to \CD'$
together with a natural isomorphism $t\colon\id_{\CD'} \cong \bar{\tau}^2$.
Concretely, for an equivariant object $(A,\phi)$ we define
\[ t^{(A,\phi)} \colon (A,\phi) \xrightarrow{\phi_{\tau^2}} \bar{\tau}^2 (A,\phi) = (\tau^2 A, \tau^2 \phi). \]
We also have an equivalence $\chi \colon \CD' \to \CD'$ of order $2$ obtained by twisting with the non-trivial character $\chi$ of $\BZ_2$. The automorphisms $\bar{\tau}$ and $\chi$ commute canonically and hence the composition $\chi \circ \bar{\tau}$ is of order $2$ in $\Aut \CD'$. 
Suppose also that\footnote{This is automatically satisfied in many instances,
see Theorem~\ref{prop:faithful action2} and Lemma~\ref{lem:number components}.}
\[ \Hom(\id_{\CD'}, \id_{\CD'}) = \BC \id. \]

\vspace{4pt}
\noindent
\textbf{Claim.} The subgroup $\BZ_2 \subset \Aut \CD'$ generated by $g = \chi \circ \bar{\tau}$ does not act on $\CD'$.

\begin{proof}[Proof of Claim]
Since $\Hom(\id_{\CD'}, \id_{\CD'}) = \BC$, any isomorphism $\theta_{g,g} \colon g^2 \xrightarrow{\cong} \id$ is
a scalar multiple of $t^{-1}$. Hence it is enough to show that $gt \neq tg$.
For any $(A,\phi)$ the map $(gt)^{(A,\phi)}$ is obtained by applying $\chi \bar{\tau}$ to
$\phi_{\tau^2} \colon (A,\phi) \to \bar{\tau}^2 (A,\phi)$. Hence it is equal to 
$\tau \phi_{\tau^2}$ (twisting by $\chi$ acts by the identity on morphisms).
On the other hand we have
\[ (t g)^{(A,\phi)} = t^{g (A,\phi)} = t^{(\tau A, \chi(\tau^2) \tau \phi)}
=
\chi(\tau^2) \tau \phi_{\tau^2}
=
-\tau \phi_{\tau^2}. \qedhere \]
\end{proof}

To translate the above into a simple concrete case, let $\tau \colon \Coh(E) \to \Coh(E)$ be the translation by a $4$-torsion point on an elliptic curve $E$. The equivariant category $\CD_{\langle \tau^2 \rangle}$ is then equivalent to $\Coh(E')$ with $E' = E/\langle \tau^2 \rangle$ being the quotient.
The induced morphism $\bar{\tau}$ is equivalent to translation $t_a$ by a $2$-torsion point
and $\chi$ is equivalent to tensoring with a $2$-torsion line bundle $\CL_{b} = \CO_{E'}(b-0_{E'})$
corresponding to a $2$-torsion point $b$ which is distinct from $a$.
We find that
the involution $\CL_b \otimes t_a^{\ast}( - )$ does not define an action of $\BZ_2$ on $\Coh(E')$,
but only a $\BZ_4$-action.

\section{Decompositions of equivariant categories} \label{sec:decomp equiv}
In this section we assume $\CD$ to be a $\BC$-linear triangulated category.

The goal of this section is to understand the decomposition of the equivariant category into its \emph{components} (defined in Section~\ref{subsec:ortho decomp}).
We will say that an action of a finite group $G$ on a category $\CD$
is \emph{faithful} if the equivariant category $\CD_G$ can not be decomposed in a non-trivial way into components, or in other words if it is indecomposable. 
Our first aim is to describe the number of components of the equivariant category
in terms of cohomology.
This leads to a useful criterion for an action to be faithful.
Then we describe each component of the equivariant category
as the equivariant category of a faithful action.
The main tool we will use are stability conditions\footnote{Often an appropriate notion
of weak stability conditions would suffice for our applications.
For simplicity we work with stability conditions throughout.}.

\subsection{Orthogonal decompositions} \label{subsec:ortho decomp}
A triangulated category $\CD$ is the \emph{orthogonal direct sum} of $n$ full subcategories $\CD_i$
if every object $E \in \CD$ is isomorphic to a direct sum $\oplus_i E_i$ with $E_i \in \CD_i$ and there are no non-trivial morphisms between objects which lie in different subcategories.  
In this case we write $\CD = \oplus_i \CD_i$.
The category $\CD$ is \emph{indecomposable} if in any such decomposition all except one summand is trivial.
Given a finite decomposition 
\[ \CD=\bigoplus_i \CD_i, \]
where all $\CD_i$ are non-trivial and indecomposable,
the summands $\CD_i$ are unique up to permutation and called the components of $\CD$.

\begin{example}
If $\CD$ is the derived category $D^b(X)$ of a smooth projective (not necessarily connected) variety $X$, then we have the orthogonal decomposition
\[ D^b(X) = \bigoplus_{i} D^b(X_i) \]
where $X_i$ are the connected components of $X$.
Hence $X$ is connected if and only if $D^b(X)$ is indecomposable. \qed
\end{example}

\subsection{Triangulated categories}
Let $G$ be a finite group acting on a $\BC$-linear triangulated category $\CD$.
We define a shift functor $[1] \colon \CD_G \to \CD_G$ by
\[ (E,\phi)[1]=(E[1], \phi[1]) \]
and we say a triangle in $\CD_G$ is distinguished
if and only it it is distinguished after applying
the forgetful functor $p$.
By a result of Elagin \cite[Thm.\ 6.10]{Elagin},
if $\CD$ admits a dg-enhancement\footnote{We refer
to \cite{CS} for references on dg-enhancements.}, then these
definitions make $\CD_G$ a triangulated category.

The existence of a dg-enhancement is a technical condition \cite{CS}
but known for the case we are most interested in.
Indeed, 
by a result of Lunts and Orlov \cite{LO}
the bounded derived category $D^b ( \Coh(X) )$
of coherent sheaves on a smooth projective variety $X$
has up to equivalence a unique dg-enhancement.

We can also ask for the stronger condition
that the $G$-action on $\CD$
lifts to a $G$-action on a dg-enhancement of $\CD$.
This is satisfied for example
if $G$ preserves a full abelian subcategory $\CA \subset \CD$
such that $D^b(\CA) \cong \CD$,
see \cite{CS} and \cite[Sec.\ 2.1]{BeckmannOberdieckModuli}.
This condition is useful since it will allow us to use methods from
Hochschild cohomology.

From now on we will always assume that $\CD_G$ is triangulated. %has a dg-enhancement.

\subsection{$\rho$-trivial actions}
\label{exmp:trivial action}
Let $\CD$ be a triangulated category with $\Hom(\id_{\CD}, \id_{\CD}) = \BC$.
Let $(\rho, \theta)$ be the action of a finite group $G$ on $\CD$
such that 
%$\rho_g = \id$ for all $g \in G$,
\[ \rho_g = \id \text{  for all  }g \in G, \]
but with $\theta$ arbitrary.
By Theorem~\ref{lem:obstruction to action} there is an associated cocycle
\[ \alpha \in H^2(G, \BC^{\ast}) \]
where the trivial action corresponds to the trivial class.

Let $f_i \colon G \to \mathrm{GL}(V_i)$
be the projective irreducible representations of $G$ of class $\alpha$,
see \cite{Cheng} for an introduction to the theory of these representations.

\begin{lemma} \label{lem:decomposition trivial action}
If $\CD_G$ is also triangulated, then we have the orthogonal decomposition
\[ \CD_G = \bigoplus_{i} \CD \otimes V_i \]
where each $\CD\otimes V_i$ is the full subcategory of pairs %\footnote{Do we need to say generated by such pairs? Or are these really all the elements?} 
 $(E \otimes V_i, \phi)$ with $E \in \CD$ and
the linearization is defined by $\phi_h = \id_E \otimes f_i(h) \colon E \otimes V_i \to E \otimes V_i$.
\end{lemma}
\begin{proof}
The case of the trivial action ($\alpha = 0$) can be found in \cite[Prop.\ 3.3]{KP}.
The argument for the general case is completely parallel (note that for projective representations we also have that the regular (projective) representation $\BC[G]$ decomposes into $\oplus V_i^\vee \otimes V_i$, see \cite[Cor.\ 3.11]{Cheng}).
\end{proof}

\subsection{Stability conditions} \label{subsec:stab_conditions}
A (Bridgeland) stability condition on a triangulated category $\CD$ is a pair $(\CA, Z)$ consisting of
\begin{itemize}
\item the heart $\CA \subset \CD$ of a bounded $t$-structure on $\CD$ and
\item a stability function $Z\colon K(\CA) \to \BC$
\end{itemize}
satisfying 
the existence and uniqueness of Harder--Narasimhan filtrations, positivity and the support property,
see \cite{BridgelandStab}.%\footnote{Do we need the support property here?}

Given an equivalence $\Phi \colon \CD \to \CD'$
of triangulated categories
the image of $\sigma$ under $\Phi$ is defined by
\[ \Phi \sigma = (\Phi \CA, Z \circ \Phi_{\ast}^{-1}) \]
where $\Phi_{\ast} \colon K(\CD) \to K(\CD')$ is the induced map on $K$-groups.
If $\Phi \colon \CD \to \CD$ is an auto-equivalence, we say that 
$\Phi$ preserves (or fixes) $\sigma$ if $\Phi \sigma = \sigma$.

Let us assume that $\CD$ 
has an action of a finite group $G$ which fixes a stability condition $\sigma = (\CA,Z)$.
Let us also assume as usual that $\CD_G$ is triangulated.
Then an application of \cite[Thm.\ 2.14]{MMS} shows that the pair
\[ \sigma_G = (\CA_G, Z_G), \quad Z_G \coloneqq Z \circ p_{\ast} \colon K(\CA_G) \to \BC \]
defines a stability condition on $\CD_G$.

If an element $E \in \CA_G$ is destabilized by $F$, then $p(E)$ is destabilized by $p(F)$.
Similarly, if $p(E)$ is destabilized by $F' \in \CA$, then the image of the adjoint morphism 
%$q(F') = (\oplus_{g} g F', \phi_{\mathrm{can}}) \rightarrow E$ 
$qF' \to E$
destabilizes $E$.
Hence an element in $(E,\phi) \in \CA_G$ is $\sigma_G$-semistable if and only if $E \in \CA$ is $\sigma$-semistable.
Being stable is more subtle:
Since there may be destabilizing objects in $\CA$ which do not lie in the image of $\CA_G$,
an equivariant object 
$(E,\phi) \in \CA_G$ can be $\sigma_G$-stable while $p(E)$ is not.
The precise relation is given by the following.
\begin{lemma} \label{Lemma_stability}
Let $(E,\phi) \in \CA_G$.
\begin{enumerate}
\item[(i)] If $E$ is $\sigma$-stable, then $(E,\phi)$ is $\sigma_G$-stable.
\item[(ii)] If $(E,\phi)$ is $\sigma_G$-stable, then 
\[ E=\bigoplus_{g\in G/H} \phi_g^{-1}(gF) \]
for some subgroup $H\subset G$ and $\sigma$-stable object $F$.
Hence $E$ is $\sigma$-polystable, and it is $\sigma$-stable if and only if it is simple.
\end{enumerate}
\end{lemma}
\begin{proof}
By applying the forgetful functor,
any destabilizing object of $(E,\phi)$ yields a destabilizing object of $E$.
This shows (i).

For the proof of (ii) compare also with \cite[Sec.\ 2.3]{Zowislok} and \cite[Lem.\ 5.9]{Nuer}. 

Assume $(E,\phi)$ is $\sigma_G$-stable and $E$ is strictly $\sigma$-semistable.
Take a $\sigma$-stable destabilizing subobject $F\subset E$. Consider the two subobjects $gF$ and $\phi_g F$ of $gE$. Since they are both stable of the same phase, either $gF=\phi_g F$ or $gF \cap \phi_g F=0$ as subobjects of $gE$.
Let $H \subset G$ denote the subgroup consisting of elements $h$ satisfying $hF=\phi_h F$.
%That this is indeed a subgroup follows from the cocycle condition of $\phi$. 
%
Define the subobject 
\[
F'\coloneqq \sum_{g \in G} \phi_g^{-1}(gF) \subset E
\]
which by definition of $H$ is equal to $F'= \oplus_{g\in G/H} \phi_g^{-1}(gF)$. The linearization $\phi$ of $E$ restricts to give a linearization 
\[
\phi_g'=\phi_g|_{F'} \colon F' \rightarrow F'
\]
making $(F',\phi')$ a subobject of $(E,\phi)$ in $\CA_G$ of the same slope. Stability of $(E,\phi)$ implies $E=F'$ which finishes the proof. 
\end{proof}

\subsection{Number of components}
We write $\Hom(f,g)$ for the vector space of a natural transformations
$f \to g$ for functors $f,g \colon \CD \to \CD$.
The number of components
of a triangulated category can be described as follows.

\begin{lemma} \label{lem:number components}
Let $\CD$ be a triangulated category with a stability condition $\sigma$.
If $\Hom_{\CD}( \id_{\CD}, \id_{\CD})$ is finite-dimensional,
then $\CD$ has finitely many components, in which case the number of components of $\CD$ is equal to 
$\dim_\BC \Hom_{\CD}( \id_{\CD}, \id_{\CD})$.
\end{lemma}
\begin{proof}
If $\CD$ admits an orthogonal decomposition $\oplus_{i=1}^{n} \CD_i$
then the projections to the $i$-th factor define $n$ linearily independent
elements in $\Hom_{\CD}( \id_{\CD}, \id_{\CD})$.
This shows that if $\Hom_{\CD}( \id_{\CD}, \id_{\CD})$ is finite-dimensional of dimension $n$,
then there are at most $n$ components $\CD_i$.

Hence it suffices to show that if $\CD$ is indecomposable,
then $\Hom_{\CD}( \id_{\CD}, \id_{\CD}) = \BC \id$.
Let $t \colon \id_{\CD} \to \id_{\CD}$ be a natural transformation.
For every stable object $A \in \CD$ we have $t^A = \lambda_A \id_A$ for some $\lambda_A \in \BC$.
Moreover, if $t^A = \lambda_A \id$ while $t^B = \lambda_B \id$ for $\sigma$-stable objects $A,B$
with $\lambda_A \neq \lambda_B$, then $\Hom(A,B) = 0$.
Let $\CD_{\lambda}$ be the full triangulated subcategory of $\CD$ generated by
all $\sigma$-stable objects $A$ for which $t^A = \lambda \id$.
The categories are orthogonal and, since $\CD$ is indecomposable, all except one are trivial.
This shows the claim.
\end{proof}

To apply Lemma~\ref{lem:number components} to the equivariant category $\CD_G$
we need to understand the $\BC$-dimension of $\Hom(\id_{\CD_G},\id_{\CD_G})$.
This is provided by the following.

\begin{prop}
\label{prop:hom_components_equiv}
Let $(\rho,\theta)$ be a $G$-action on a triangulated category $\CD$ which lifts to an action on a dg-enhancement of $\CD$. Then there exists an isomorphism
\begin{equation} \label{idDG idDG}
\Hom(\id_{\CD_G}, \id_{\CD_G}) \cong \left( \bigoplus_{g\in G} \Hom(\id_\CD, \rho_g) \right)^G
\end{equation}
where the action on the right is given by conjugation.
\end{prop}

More precisely, the $G$-action is given as follows. An element $g \in G$ acts by
\[ g \colon \Hom(\id_{\CD}, \rho_h) \to \Hom(\id_{\CD}, \rho_{ghg^{-1}}) \]
by sending $t \colon \id_{\CD} \to \rho_h$ to $g \bullet t$ which is defined by the commutative diagram
\begin{equation} \label{isom conjugation}
\begin{tikzcd}
\rho_g \rho_{g^{-1}} \ar{r}{\rho_g t \rho_{g^{-1}}} \ar{dd}{\theta_{g,g^{-1}}}
& \rho_g \rho_h \rho_{g^{-1}} \ar{d}{\theta_{g,h} \rho_{g^{-1}}} \\
& \rho_{gh} \rho_{g^{-1}} \ar{d}{\theta_{gh, g^{-1}}} \\
\id_{\CD} \ar{r}{g \bullet t} & \rho_{ghg^{-1}}.
\end{tikzcd}
\end{equation}

\begin{proof}[Proof of Proposition~\ref{prop:hom_components_equiv}]
The left hand side of \eqref{idDG idDG} 
can be identified with the degree $0$ Hochschild cohomology of $\CD_G$
and the result hence follows from the description of the Hochschild cohomology of $\CD_G$ by Perry, see \cite{Perry} and Theorem~\ref{lem:perry} below.

However, to provide an idea of the proof, let us nevertheless sketch the argument
in the case where $\CD$ is an abelian category.
The argument used to prove the result of Perry is parallel
by translating all steps into the language of dg-categories.

Hence let $\CD$ be an abelian category.
By Proposition~\ref{prop:UP2} and its proof we have
\[ \Hom(\id_{\CD_G}, \id_{\CD_G}) \cong \Hom( \id_{\CD_G} q, \id_{\CD_G} q )^G \]
where $g$ acts on $\Hom(q,q)$ by sending $t \colon q \to q$
to the composition $q \cong qg \xrightarrow{tg} qg \cong q$
(the isomorphisms are provided by $\theta$).
By Proposition~\ref{prop:UP1} we have further
\[ \Hom( \id_{\CD_G} q, \id_{\CD_G} q ) \cong \Hom( p\id_{\CD_G} q, p\id_{\CD_G} q )^G \]
where $g \in G$ acts by sending $t \colon pq \to pq$ to $pq \cong gpq \xrightarrow{gt} gpq \cong pq$.
This yields
\[ \Hom(\id_{\CD_G}, \id_{\CD_G}) \cong \Hom\left( \oplus_{g \in G} \rho_g,\, \oplus_{g \in G} \rho_g \right)^{G \times G} \]
where $(g_1, g_2)$ acts via $t \mapsto g_1 t g_2$.
Finally observe that we have
\[ \bigoplus_{g \in G} \rho_g = \Ind_{G}^{G \times G} \id_{\CD} \]
where the embedding $G \hookrightarrow G \times G$ is given by $g \mapsto (g,g^{-1})$.
This yields the claim by the adjunction with the restriction functor $\mathrm{Res}_{G}^{G \times G}$:
\[
\Hom\left( \oplus_{g \in G} \rho_g,\, \oplus_{g \in G} \rho_g \right)^{G \times G}
\cong
\Hom\left( \id_{\CD}, \oplus_{g \in G} \rho_g \right)^{G}.\]
\end{proof}

\subsection{Faithful actions I}
Consider a $G$-action on $\CD$ such that $\CD_G$ is triangulated.
We say that the $G$-action on $\CD$ is \emph{faithful} if the associated equivariant category is indecomposable.
We have the following criterion for an action to be faithful.

\begin{thm} \label{prop:faithful action2}
Let $(\rho, \theta)$ be a $G$-action on an indecomposable triangulated category $\CD$
which preserves a stability condition $\sigma$
and lifts to an action on a dg-enhancement of $\CD$.
If $\rho_g \not\cong \id$ for all $g \neq 1$, then $G$ acts faithfully.
\end{thm}

We refer to Section~\ref{subsec:equiv cat of not act example} below for an example
which shows that the converse of Theorem~\ref{prop:faithful action2} does not hold.

The proof relies on the following lemma.
\begin{lemma} \label{lem:t isomorphism}
Let $\CD$ be an indecomposable triangulated category with an action by a finite group $G$ which fixes a stability condition $\sigma$.
For any $g \in G$, if $t \colon \id \to \rho_g$ is a natural transformation, then either $t = 0$ or $t$ is an isomorphism.
\end{lemma}
\begin{proof}
If $A$ is a $\sigma$-stable object, then so is its image $gA$, and in this case
by stability the morphism $t_A \colon A \to gA$
is either an isomorphism or zero.
Moreover, if $t_A$ is an isomorphism while $t_B=0$ for some $\sigma$-stable objects $A,B$,
then
$\Hom(A,B) = 0$.
The claim now follows by arguing as in the proof of Lemma~\ref{lem:number components}.
\end{proof}

\begin{proof}[Proof of Theorem~\ref{prop:faithful action2}]
By Lemma~\ref{lem:number components} the number of components of $\CD_G$
is the dimension of \eqref{idDG idDG}. By Lemma~\ref{lem:t isomorphism} this dimension is $1$.
\end{proof}

\subsection{Faithful actions II} 
We show that when determining equivariant categories it is enough to consider faithful actions.

\begin{thm} \label{prop:faithful action}
Let $(\rho, \theta)$ be a $G$-action on a indecomposable triangulated category $\CD$ 
which preserves a stability condition $\sigma$
and lifts to an action on a dg-enhancement of $\CD$.
Then there exists a finite decomposition
\[ \CD_G = \bigoplus_{i} \CD_i \]
and for every $i$ a faithful action $\rho_i$ by a finite group $K_i$ on $\CD$ such that $\CD_i \cong \CD_{K_i}$.
Moreover, every $\rho_i$ preserves the stability condition $\sigma$.

If $G$ is abelian, then there exists a subgroup $H \subset G$ such that
for every $i$ we can take $K_i = G/H$ and the map $G \to \Aut \CD$ factors through $G/H \xrightarrow{\rho_i} \Aut \CD$.
\end{thm}

For the proof we will need the following lemma of independent interest.

\begin{lemma} \label{lem:transitive action}
Let $G$ act on triangulated category $\CD$ with a dg-enhancement.
Let $\CD = \oplus_i \CD_i$ be an orthogonal decomposition
such that $G$ acts transitively on the set of components $\{ \CD_i \}$. % $\{ \CD_0, \ldots, \CD_n \}$. 
Let $K \subset G$ be the subgroup of elements which preserves a fixed component $\CD_0$.
Then the composition of the inclusion $(\CD_0)_K \hookrightarrow \CD_K$ and the induction functor $\Ind_K^{G} \colon\CD_K \to \CD_G$
defines an equivalence
\[ F \colon (\CD_0)_K \xrightarrow{\cong} \CD_G. \]
The composition $(\CD_0)_K \xrightarrow{F} \CD_G \xrightarrow{p} \CD$ is given by $(E,\phi) \mapsto \oplus_{g \in G/K} gE$.
\end{lemma}
\begin{proof}
Since $\CD$ has a dg-enhancement, all equivariant categories are triangulated.
The composition of $\mathrm{Res}^G_K \colon \CD_G \to \CD_K$
followed by the restriction to $(\CD_0)_K$ is the adjoint (both left and right) to $F$.
This and a straightforward calculation implies that $F$ is fully faithful and hence yields a semi-orthogonal decomposition
$\CD_G = \langle (\CD_0)_K , \CE \rangle$. Moreover, for any object $A \in \CE$ we have that the projection of $pA$
to $\CD_0$ vanishes and since $pA$ is $G$-invariant, one obtains that $pA=0$, so $A=0$.
\end{proof}

\begin{proof}[Proof of Theorem~\ref{prop:faithful action}]
Recall from Proposition~\ref{prop:hom_components_equiv} that we have
\begin{equation} \Hom_{\CD_G}(\id_{\CD_G}, \id_{\CD_G}) \cong \left( \bigoplus_{g \in G} \Hom_{\CD}(\id_{\CD}, \rho_g) \right)^G. \label{aaabc} \end{equation}
Since $G$ acts by conjugation on the right hand side, we can consider the decomposition according to conjugacy classes $\Fc$: 
\begin{equation*} 
\Hom_{\CD_G}(\id_{\CD_G}, \id_{\CD_G}) = \bigoplus_{\Fc} V_{\Fc}\ \text{ with } \ 
V_{\Fc} = \left( \bigoplus_{g \in \Fc} \Hom(\id_{\CD}, \rho_g) \right)^G.
\end{equation*}
For any non-trivial transformations $t_1 \colon \id \to \rho_g$ and $t_2 \colon \id \to \rho_g$
by Lemma~\ref{lem:t isomorphism} we can form the composition $t_2^{-1} \circ t_1$, which (since $\CD$ is indecomposable) is a multiple of the identity.
We conclude that $\dim V_{\Fc} \in \{ 0, 1 \}$ for all $\Fc$.

Let $H \subset G$ be the set of all elements whose conjugacy class $\Fc$ satisfies $\dim V_{\Fc} = 1$.
A direct check shows that $H$ is a subgroup of $G$. In particular, it is normal. Moreover for all $h \in H$ there exist isomorphisms\footnote{
The isomorphisms $\rho_g t_h \rho_{g^{-1}} \cong t_{ghg^{-1}}$ are provided by
\eqref{isom conjugation}.}
\begin{equation} t_h \colon \id \xrightarrow{\cong} \rho_h \,\text{ such that }\, \rho_g t_h \rho_{g^{-1}} \cong t_{g h g^{-1}} \text{ for all } g \in G. \label{com_rln} \end{equation}
After modifying the action by the isomorphisms $t_h$, we may assume $\rho_h = \id$ and $t_h = \id$.
Relation \eqref{com_rln} is then equivalent to the condition that the composition
\begin{equation}\label{some eqn}
\begin{tikzcd}
\id \ar{r}{\theta_{g,g^{-1}}^{-1}} 
& \rho_{g}\rho_{g^{-1}} \overset{\rho_g t_h \rho_{g}^{-1}}{=} \rho_g \rho_h \rho_{g^{-1}} \ar{r}{\theta_{g,h} \rho_{g^{-1}}} 
& \rho_{gh} \rho_{g^{-1}} \ar{r}{\theta_{gh, g^{-1}}} & \rho_{ghg^{-1}} \overset{t_{ghg^{-1}}}{=} \id
\end{tikzcd} 
\end{equation}
is the identity for all $g \in G$ and $h \in H$.

Let $V_i$, $i \in I$ be the projective irreducible representations of $H$ for the cocycle
defined by $\theta$. %the restriction of the $G$-action on $D^b(S)$ to $H$.
We consider the decomposition of Example~\ref{exmp:trivial action},
\[ \CD_H = \bigoplus_{i \in I} \CD \otimes V_i, \]
and the induced $G$-action on $\CD_H$ as in Proposition~\ref{rem:successive_quotients}.
For any $\tilde{E} = (E \otimes V_i, \phi)$ in $\CD_H$ we have
$\rho_g \tilde{E} = (\rho_g (E) \otimes V_i, \phi')$ where
\[ \phi'_h = [ \rho_g (E) \otimes V_i \xrightarrow{\phi_{g^{-1} h g}}
\rho_g \rho_{g^{-1} h g}(E) \otimes V_i \xrightarrow{(*)} \rho_{h} \rho_g(E) \otimes V_i ]
\]
and the isomorphism (*) is precisely $\rho_g$ applied to the inverse of \eqref{some eqn}. 
%The equality $\theta_{g,g^{-1}} \rho_g = \rho_g \theta_{g,g^{-1}}$ follows by associativity.
Since (*) is the identity
and $\phi_{g^{-1} h g} = \id_E \otimes f_i(g^{-1}h g)$, we find that
\[ \rho_g \tilde{E} \in \CD \otimes V_{g(i)} \]
where $V_{g(i)}$ is the irreducible representation defined by $h \mapsto f_i(g^{-1} h g)$.

This shows that for any $j \in I/G$ (where $G$ acts on the index set $I$ by $i \mapsto g(i)$)
the $G$-action on $\CD_H$ preserves and acts transitively on the subcategory
\[ \CE_j = \bigoplus_{i \in I \colon \bar{i} = j} \CD\otimes V_i. \]
Consider the quotient action $G/H$ on $\CD_H$ as in Proposition~\ref{rem:successive_quotients}.
Then $G/H$ acts also transitively on $\CE_j$. Let
$K_j \subset G/H$
be the subgroup which preserves a given (fixed) summand $\CD \otimes V_{i_j}$ of $\CE_j$.
Applying Lemma~\ref{lem:transitive action} yields the equivalence
\[ (\CE_j)_{G/H} \cong (\CD \otimes V_{i_j})_{K_j}. \]
Hence by Proposition~\ref{rem:successive_quotients} we find that
\begin{equation} \label{qqqqqq}
\CD_G \cong (\CD_H)_{G/H} \cong \bigoplus_{j \in I/G} (\CE_j)_{G/H} \cong \bigoplus_{j \in I/G} (\CD \otimes V_{i_j})_{K_j}.
\end{equation}

The order of $I$ equals the number of conjugacy classes of $H$.\footnote{This uses
that by \eqref{some eqn} all conjugacy classes of $H$ consist of '$\alpha$-elements'
where $\alpha$ is the cocycle defined by $\theta$, see \cite[Sec.\ 2]{Cheng}.}
Hence the order of $I/G$ equals the number of conjugacy classes of $G$ which are contained in $H$.
This latter number is by construction
%the number of $\Fc$ with $\dim V_{\Fc} = 1$, so 
the dimension of \eqref{aaabc},
which by Lemma~\ref{lem:number components} is precisely the number of components of $\CD_G$.
This shows that the summands $(\CD \otimes V_{i_j})_{K_j}$ in \eqref{qqqqqq} are the components of $\CD_G$, and therefore indecomposable.
Hence the action of $K_j$ on $\CD \otimes V_{i_j} \cong \CD$ is faithful.
Moreover, since the induced $G$-action on $\CD_H$ preserves the induced stability condition $\sigma_H$,
and $\sigma_H$ restricts to every $\CD \otimes V_i \cong \CD$ as $\sigma$, the $K_j$-action preserves $\sigma$.
If $G$ is abelian, then the action of $G$ on $I$ is trivial, so $K_j = G/H$ for all $j$.
\end{proof}

\subsection{A stronger version of Theorem~\ref{lem:obstruction to action}}
Using the techniques from this section, we can prove the following
stronger version of part (c) of Theorem~\ref{lem:obstruction to action}.
\begin{cor}
\label{cor:inv_simple_cyclic}
Let $\CD$ be indecomposable and let $\BZ_n \subset \Aut \CD$ such that
\begin{itemize}
\item $\BZ_n$ preserves a stability condition $\sigma = (\CA,Z)$ with $D^b(\CA) \cong \CD$, and
\item there exists a $\BZ_n$-invariant simple object in $\CD$.
\end{itemize}
Then there exists an action of $\BZ_n$ on $\CD$ such that the induced map $\BZ_n \to \Aut \CD$ is the inclusion we started with.
\end{cor}
\begin{proof}
%Assume that $G \subset \Aut M$ is cyclic, i.e.\ let $G = \BZ_n$ for some $n$.
Using part (c) of Theorem~\ref{lem:obstruction to action},
we have an action of $\BZ_{n^2}$ on $\CD$ 
such that the induced map $\BZ_{n^2} \to \Aut \CD$ is the quotient map to the given subgroup $\BZ_n \subset \Aut \CD$.

Since the $\BZ_{n^2}$-action preserves the stability condition and $D^b(\CA) \cong \CD$,
it lifts to an action on a dg-enhancement of $\CD$.
Applying Theorem~\ref{prop:faithful action}
we hence find a faithful action by some $\BZ_{m}$
such that its image in $\Aut \CD$ is the subgroup $\BZ_n$.

We will show that $m=n$. Let $k = m/n$ and consider the short exact sequence
\[ 0 \to \BZ_k \to \BZ_m \to \BZ_n \to 0. \]
The image of $\BZ_k$ in $\Aut \CD$ is the trivial group. Since $H^2(\BZ_k, \BC^{\ast}) = 0$ we see that the $\BZ_m$-action on $\CD$ restricts to the trivial action by $\BZ_k$ (or, more precisely, to an action which is isomorphic to the trivial action).
We consider the induced action of $\BZ_n$ on the equivariant category
\begin{equation} \CD_{\BZ_k} = \bigoplus_{i=1}^{k} \CD \otimes V_i, \label{dfsdf11} \end{equation}
where $V_i$ are the irreducible representations of $\BZ_k$.

Since $\CD_{\BZ_m}$ is indecomposable, $\BZ_n$ acts transitively on the summands \eqref{dfsdf11}.
Let $\BZ_l \subset \BZ_n$ be the stabilizer of the first summand.
In particular, $\BZ_k = \BZ_n / \BZ_l$.
Applying Proposition~\ref{rem:successive_quotients} and Lemma~\ref{lem:transitive action} we obtain the equivalence
\[ 
\CD_{\BZ_m} \cong \left( \CD_{\BZ_k} \right)_{\BZ_{n}} 
\cong  \left( \bigoplus_{i} \CD \otimes V_i \right)_{\BZ_{n}}
\cong (\CD \otimes V_1)_{\BZ_{\ell}}.
\]
Moreover, under this equivalence the forgetful functor $p \colon \CD_{\BZ_m} \to \CD$ is given by
sending $(E \otimes V_1,\phi) \in (\CD \otimes V_1)_{\BZ_{\ell}}$ to
%\[ 
\begin{equation} \label{pZk}
p_{\BZ_k} \left( \bigoplus_{g \in \BZ_n/\BZ_l} gE \otimes V_1 \right)
= \bigoplus_{g \in \BZ_n/\BZ_l} g E.
\end{equation}

On the other hand, by the assumption that $\BZ_n \subset \Aut \CD$ fixes some simple object,
there is a simple object $F \in \CD$ which is invariant under the $\BZ_m$-action.
By Lemma~\ref{rmk:Ploog_simple} the object $F$ admits a linearization with respect to $\BZ_m$ and hence lies in the image of the forgetful functor from the equivariant category $\CD_{\BZ_m}$.
In other words, $F$ is of the form \eqref{pZk}.
This implies that $\BZ_n/\BZ_l$ is trivial, so $k=1$.
\end{proof}

\subsection{The equivariant category of Example~\ref{subsec:does not act}}
\label{subsec:equiv cat of not act example}
We return to Example~\ref{subsec:does not act}.
Recall that there we considered a subgroup
\[ \BZ_2 \subset \Aut \CD' \]
generated by an involution $g \colon \CD' \to \CD'$
which does not act on $\CD'$. Note that by Corollary~\ref{cor:inv_simple_cyclic} this implies that there is no simple object in $\CD'$ which is invariant under the action of $\BZ_2$. 

By part (c) of Theorem~\ref{lem:obstruction to action}
the involution $g$ defines a (unique) $\BZ_4$-action on $\CD'$.
Here we determine the associated equivariant category.
Let us assume that the action preserves a stability condition
and lifts to a dg-enhancement of $\CD'$.

\vspace{4pt}
\noindent \textbf{Claim.} We have an equivalence $\Phi \colon \CD' \xrightarrow{\cong} \CD'_{\BZ_4}$.
The composition of $\Phi$ with the forgetful functor is given by $(p \circ \Phi)(E) = E \oplus gE$.

\begin{proof}[Proof of Claim 2]
We first show that the equivariant category $\CD'_{\BZ_4}$ is indecomposable:
By Proposition~\ref{prop:hom_components_equiv} and Lemma~\ref{lem:t isomorphism}
this reduces to showing that 
\[ \Hom(\id, \rho_{g^2})^{\BZ_4} = 0. \]
However, we have seen in Example~\ref{subsec:does not act} that the generator
of this vector space
\[ t \colon \id \xrightarrow{\cong} \rho_{g^2} \]
is not $G$-invariant. Hence $\Hom(\id, \rho_{g^2})^{\BZ_4} = 0$ and $\CD'_{\BZ_4}$ is indecomposable.

Since $H^2(\BZ_2, \BC^{\ast}) = 0$, the restriction of the $\BZ_4$-action to the subgroup $\BZ_2 \subset \BZ_4$ is trivial.
An application of Proposition~\ref{rem:successive_quotients} 
gives
\[ \CD'_{\BZ_4} \cong \left( \CD'_{\BZ_2} \right)_{\BZ_4/\BZ_2} = \left( \CD' \oplus \CD' \right)_{\BZ_4/\BZ_2}. \]
Since $\CD'_{\BZ_4}$ is indecomposable, $\BZ_4/\BZ_2$ acts transitively on the two summands.
The claim now follows from Lemma~\ref{lem:transitive action}.
\end{proof}

Recall the concrete example discussed at the end of Section~\ref{subsec:equiv cat of not act example}. 
We have seen that the involution $\CL_b \otimes t_a^{\ast}( - )$ on $D^b(E')$
does not define an action of $\BZ_2$,
but only a $\BZ_4$-action.
By the above claim we see that
\[ D^b(E')_{\BZ_4} \cong D^b(E'). \]
%\[ \Coh(E)_{\BZ_4} \cong \Coh(E). \]
Another example is described in \cite[Sec.\ 7]{BeckmannOberdieckModuli}.

\section{The Serre functor on equivariant categories} 
\label{subsec:Serrefunctor}
Let $\CD$ be a $\BC$-linear triangulated category with finite-dimen\-sional Hom's.

A Serre functor for $\CD$ is an equivalence $S \colon \CD \to \CD$ together with
a collection of bifunctorial isomorphisms %for any objects $A,B$,
\[ \eta_{A,B} \colon \Hom(A,B) \xrightarrow{\cong} \Hom(B, SA)^{\vee} \]
for all objects $A,B$.
We write 
$\langle f,f' \rangle$
for the pairing of $f \in \Hom(A,B)$ with $f' \in \Hom(B,SA)$.
The functoriality in $A$ is equivalent to
\begin{equation} \label{func_A}
\langle f \circ \psi, f' \rangle = \langle f, S \psi \circ f' \rangle
\end{equation}
for all $f \in \Hom(A',B)$, $\psi \colon A \to A'$ and $f' \in \Hom(B, SA)$.
The functoriality in $B$ is equivalent to
\begin{equation} \label{func_B} \langle \rho \circ f, f' \rangle = \langle f, f' \circ \rho \rangle \end{equation}
for all $f \in \Hom(A,B)$, $f' \in \Hom(B', SA)$ and $\rho \colon B \to B'$.
The following is well-known.
%The following can be proven along the lines of \cite[Prop.\ 1.26]{HuybrechtsFMbook}.
\begin{lemma} \label{lemma123}
For every equivalence $F \colon \CD\to \CD$ there exist a canonical 2-isomorphism $t_F \colon SF \xrightarrow{\cong} FS$ with the following properties:
\begin{enumerate}
\item[(a)] For all equivalences $F,G \colon \CD \to \CD$ the following diagram commutes:
\[
\begin{tikzcd}
S F G \ar{r}{t_F G} \ar[bend right]{rr}{t_{FG}} & F S G \ar{r}{F t_G} & F G S.
\end{tikzcd}
\]
\item[(b)] $\langle f,f' \rangle = \langle F f, (t_{F}^A)^{-1} \circ F f' \rangle$ for all $f \in \Hom(A,B)$ and $f' \in \Hom(B,SA)$.

\end{enumerate}
\end{lemma}
\begin{proof}
For any $A,B \in \CD$ we have the chain of isomorphisms
\begin{multline}\label{eq:comp_serre}
\Hom(B, SFA) \cong \Hom( FA, B)^{\vee} \cong \Hom(A, F^{-1} B)^{\vee} \\ \cong \Hom( F^{-1} B, SA) \cong \Hom( B, F S A)
\end{multline}
where we applied Serre duality in the first and third, and $F$ and $F^{-1}$ in the second and fourth step respectively.
Since Serre duality and application of $F$ is functorial in both arguments, the isomorphisms are functorial in both $A$ and $B$.
By the Yoneda lemma this gives the desired natural transformation $t_F$.
For the functoriality of $t$ in $F$ we need to show that for every $A \in \CD$ we have
the commutative diagram
\[
\begin{tikzcd}
S F G (A) \ar{r}{t_F(G A)} \ar[bend right]{rr}{t_{FG}(A)} & F S G(A) \ar{r}{F(t_G(A))} & F G S(A).
\end{tikzcd}
\]
This is checked by applying $\Hom(B, - )$. The adjunctions used to define the composition yield precisely the adjunction to define the curved arrow. This shows (a). 

For (b) we dualize (\ref{eq:comp_serre}) and replace $B$ with $FB$, i.e.\
\begin{multline*}
\Hom(FB, SFA)^\vee \cong \Hom( FA, FB) \cong \Hom(A, B) \\ \cong \Hom(B, SA)^\vee \cong \Hom( F B, F S A)^\vee.
\end{multline*}
If we apply this chain of isomorphisms to $\eta_{FA,FB}(Ff)$, we obtain
\[ \eta_{FA,FB}(Ff) \mapsto Ff \mapsto f \mapsto \eta_{A,B}(f) \mapsto F\eta_{A,B}(f). \]
By definition, the resulting dual map $\tilde{f} \coloneqq F\eta_{A,B}(f)$ satisfies
\[
\tilde{f} (g) =\langle f,F^{-1}g\rangle
\]
for all $g\in \Hom(FB,FSA)$. Applied to $g=Ff'$ this yields
\[
\tilde{f} (Ff') = \langle f,f'\rangle.
\]
On the other hand, we have by construction $\tilde{f} = ((t_F^A)^{-1})^{\vee} \left( \eta_{FA,FB}(Ff)\right)$. Thus
\begin{align*}
\langle Ff,(t_F^A)^{-1} \circ Ff'\rangle &= \eta_{FA,FB}(Ff) ((t_F^A)^{-1} \circ Ff')\\
&= \left( ((t_F^A)^{-1})^{\vee} \left( \eta_{FA,FB}(Ff) \right) \right) (Ff')\\
&=\tilde{f}(Ff')\\
&=\langle f,f'\rangle
\end{align*}
as desired.
\end{proof}

Let $G$ be a finite group acting on $\CD$.
Given an object $(A, \phi)$ in $\CD_G$ we claim that the collection
\begin{equation} \label{new_G_lin}
\phi'_g \colon SA \xrightarrow{S \phi_g} S gA \xrightarrow{t_g^A} g SA
\end{equation}
for all $g \in G$ 
is a $G$-linearization of $SA$.
Indeed, consider the diagram
\[
\begin{tikzcd}
SA \ar{r}{S\phi_g} \ar[swap]{dr}{S\phi_{gh}} & SgA \ar{d}{Sg\phi_h} \ar{r}{t_g^A} & g SA \ar{d}{g S\phi_h} \\
& Sgh A \ar{r}{t_g^{hA}} \ar[swap]{dr}{t_{gh}^A} & g ShA \ar{d}{g t_h^A} \\
& & gh SA.
\end{tikzcd}
\]
The relation we need to check (given in \eqref{compatibility}) is the commutativity of the outer triangle.
Since $t_g$ is a natural transformation the upper right square commutes.
The upper left triangle commutes since it is obtained by applying $S$ to the diagram \eqref{compatibility} for the linearization $\phi$.
The lower right triangle commutes by the functoriality of $t_F$ in $F$ implied by Lemma~\ref{lemma123}.

For any morphism $(A, \phi) \to (B,\psi)$ given by a morphism $\alpha \colon A \to B$ in $\CD$ %(commuting with the $G$-action) 
the morphism $S \alpha \colon SA \to SB$ is $G$-invariant with respect to the the $G$-linearizations of $SA$ and $SB$ just defined. Hence we obtain an equivariant morphism $S\alpha \colon (SA, \phi') \to (SB,\psi')$. This yields a functor
\[ \tilde{S} \colon \CD_G \to \CD_G, \ (A,\phi) \mapsto (SA, \phi'), \ \alpha \mapsto S(\alpha) \]
which by construction satisfies $p \tilde{S} = S p$.

\begin{prop}
The functor $\tilde{S}$ together with the restriction of $\eta_{A,B}$ to the $G$-invariant part
defines a Serre functor on $\CD_G$.
Equivalently for any two objects $(A,\phi)$ and $(B, \psi)$ in $\CD_G$ we have bifunctorial isomorphisms
\[ \eta_{A,B} \colon \Hom(A,B)^G \xrightarrow{\cong} ( \Hom( B, SA )^G )^{\vee} \]
where the $G$-action on the left is defined by the linearizations $\phi, \psi$
and the $G$-action on the right is defined by the linearizations $\psi$ and $\phi'$ (as in \eqref{new_G_lin}).
\end{prop}

\begin{proof}
The action of $g\in G$ on $\Hom(B,SA)$ defined by $\psi$ and $\phi'$ is %$\phi' = t_g^A \circ S \phi_g$ is
\[ f' \mapsto S (\phi_g)^{-1} \circ (t_{g}^A)^{-1} \circ gf' \circ \psi_g. \]
Hence for any $f \in \Hom(A,B)$ we obtain
\begin{align*} \langle f, g \cdot f' \rangle 
& = \langle f,  S (\phi_g)^{-1} \circ (t_{g}^A)^{-1} \circ gf' \circ \psi_g \rangle\\
& = \langle \psi_g \circ f \circ \phi_g^{-1}, (t_{g}^A)^{-1} \circ gf' \rangle\\
& = \langle g^{-1} \psi_g \circ g^{-1} f \circ g^{-1} (\phi_g^{-1}),\, (t_{g^{-1}}^{gA})^{-1} \circ g^{-1}(t_{g}^A)^{-1} \circ f' \rangle \\
& = \langle (\psi_{g^{-1}})^{-1} \circ g^{-1} f \circ \phi_{g^{-1}},\, f' \rangle \\
& = \langle g^{-1} \cdot f, f' \rangle
\end{align*}
where we have used \eqref{func_A} and \eqref{func_B} in the second and Lemma~\ref{lemma123} (b) and (a) in the third and fourth equality respectively.
Hence the action on $\Hom(B,SA)$ with respect to $\psi, \phi'$ is dual to the action on $\Hom(A,B)$ with respect to $\phi, \psi$.
This implies the claim.
\end{proof}

%(ii) Another approach to the Serre functor $\tilde{S}$
%is to apply Neeman's criterion \cite{Neeman} to show that $\hom_{\pi}(A, - ) \colon D(X)_G \to D(T)_G$ has a right adjoint.

\section{Hochschild cohomology} \label{sec:cohomology}
In this section we discuss how Hochschild cohomology can be used to describe equivariant categories.
The definition requires us to work with enhanced categories, for which we take as model dg-categories (another, in characteristic $0$ equivalent, choice would be stable $\infty$ categories
as used in the work of Perry \cite{Perry}).

Throughout let $\CC$ be a pre-triangulated dg-category over $\BC$.

\subsection{Definition} Let $\mathrm{Func}(\CC, \CC)$ be the dg-category of functors $\CC \to \CC$ where morphisms are natural transformations.
The Hochschild cohomology of $\CC$ with values in a functor $\phi \colon \CC \to \CC$ is defined by
\[ \HH^{\bullet}(\CC, \phi) \coloneqq \bigoplus_{i \in \BZ} \Hom_{\mathrm{Func}(\CC, \CC)}( \id_{\CC}, \phi[i]) [-i]. \]
The (absolute) Hochschild cohomology of $\CC$ is
\[ \HH^{\bullet}(\CC) = \HH^{\bullet}(\CC, \id_{\CC}). \]

\subsection{Equivariant category}
Every equivalence $F \colon \CC\to \CC$ induces (functorially) a morphism on Hochschild cohomology by conjugation,
\[ F_{\ast} \colon ( \id_{\CC} \xrightarrow{t} \id_{\CC}[i] ) \mapsto ( F \id_{\CC} F^{-1} \xrightarrow{F t F^{-1}} F \id_{\CC}[i] F^{-1}) \cong (\id_{\CC} \xrightarrow{F t F^{-1}} \id_{\CC}[i]). \]
In particular, a group action on $\CC$ induces a group action on Hochschild cohomology.

By work of Perry, we have the following description of the Hochschild cohomology of the equivariant category $\CC_G$.
\begin{thm} \textup{(}\cite[Thm.\ 4.4]{Perry}\textup{)} \label{lemma124} \label{lem:perry}
Let a finite group $G$ act on $\CC$ via the equivalences $\rho_g \colon \CC \to \CC$ for all $g \in G$. Then we have
\begin{equation}
\HH^{\bullet}(\CC_G) = \left( \bigoplus_{g \in G} \HH^{\bullet}(\CC, \rho_g) \right)^G \label{hh:equivaraint cat}
\end{equation}
where $G$ acts on the right by conjugation.
\end{thm}

%In categorical language the group action $\rho$ on $\CC$ is a $2$-representation of $G$.
The cohomology group
\[ \TTr_{\CC}(g) = \HH^{\bullet}(\CC, \rho_g) = \bigoplus_{i} \Hom_{\mathrm{Func}(\CC, \CC)}( \id_{\CC}, \rho_g[i]) [-i] \]
is called the categorical trace of $g$ with respect to the given $G$-action \cite{GK}.
If the element $h \in G$ commutes with $g$, then we have an induced action of $h$ on $\TTr_{\CC}(g)$. The $2$-characters of the representation are
\[ \chi_{\rho}(g,h) = \Tr( h | \TTr_{\CC}(g) ) \]
where the trace is taken in the supercommutative sense.\footnote{In physics language these are the $h$-twisted, $g$-twined characters of the representation.}
%Lemma~\ref{lemma124} implies the following.
\begin{cor}
Assume $\HH^{\bullet}(\CC, \rho_g)$ is finite-dimensional for all $g$. Then
\[ e( \HH^{\bullet}(\CC_G) ) = \frac{1}{|G|} \sum_{gh=hg} \chi_{\rho}(g,h). \]
\end{cor}
\begin{proof}
If a finite group $G$ acts on a vector space $V$, then
\[ \dim V^G = \frac{1}{|G|} \sum_{g \in G} \Tr(g | V). \]
Applying this to the expression in Theorem~\ref{lemma124} yields the claim.
%we find
%\[
%e( \HH^{\bullet}(\CC_G) )
%= \frac{1}{|G|} \sum_{h \in G} \Tr\left( h \middle| \bigoplus_{g \in G} \HH^{\bullet}(\CC, \rho_g) \right)
%\]
%Under the action of $h$ the summand $\HH^{\bullet}(\CC, \rho_g)$ gets send to $\HH^{\bullet}(\CC, \rho_{hgh^{-1}})$
%hence only those $g$ with $hgh^{-1} = g$ contribute to the trace. Thus
%\begin{align*}
%e( \HH^{\bullet}(\CC_G) )
%&
%= \frac{1}{|G|} \sum_{h \in G} \Tr\left( h \middle| \bigoplus_{g \in G}^{hg = gh} \HH^{\bullet}(\CC, \rho_g) \right) \\
%& = \frac{1}{|G|} \sum_{\substack{g,h \in G \\ gh = hg}} \Tr\left( h | \HH^{\bullet}(\CC, \rho_g) \right).
%\end{align*}
\end{proof}

In the decomposition \eqref{hh:equivaraint cat} an element $h$ acts by
$\HH^{\bullet}(\CC, \rho_g) \to \HH^{\bullet}(\CC, \rho_{hgh^{-1}})$.
Hence we may rewrite
\begin{equation} \HH^{\bullet}(\CC_G) = \bigoplus_{\Fc} V_{\Fc},\ \text{ where } \ V_{\Fc} = \left( \bigoplus_{g \in \Fc} \HH^{\bullet}(\CC, \rho_g) \right)^G \label{defn:Vc}\end{equation}
and $\Fc$ runs over all conjugacy classes of $G$.

Recall that the group of characters $G^{\vee}$ acts on the equivariant category $\CC_G$.
The following describes the induced action on $\HH^{\bullet}(\CC_G)$.
\begin{lemma} \label{lemma125}
For every $\chi \in G^{\vee}$ we have
$\chi|_{V_{\Fc}} = \chi(\Fc) \id_{V_{\Fc}}$.
%\HH^{\bullet}(\CC_G) & = \left( \bigoplus_{g \in G} \HH^{\bullet}(\CC, \rho_g) \right)^G
\end{lemma}

\begin{proof}
If a natural transformation $t \colon \id_{\CC} \to \id_{\CC}[i]$ is $G$-invariant,
then $t^A \colon A \to A[i]$ is $G$-invariant for every $(A,\phi) \in \CC_G$.
Hence $t$ lifts to the natural transformation
\begin{equation} \tilde{t} \colon \id_{\CC_G} \to \id_{\CC_G}[i] \label{dfsdf} \end{equation}
defined by $\tilde{t}^{(A,\phi)} \coloneqq t^A$.
%Since $\tilde{t}^{(A,\phi)} \colon (A,\phi) \to (A,\phi)[i]$ is given by $t^A \colon A \to A[i]$
Since $\tilde{t}^{(A,\phi)}$ does not depend on the linearization, but only on the underlying object $A$,
%The dual group $G^{\vee}$ acts on $\CC_{G}$ by twisting the $G$-linearization and acting by the identity on morphisms.
we find for all $\chi \in G^{\vee}$ that
\[ (\chi \tilde{t} \chi^{-1})^{(A,\phi)}
= \chi \tilde{t}^{\chi^{-1} (A,\phi)}
= \chi \tilde{t}^{(A,\chi^{-1} \phi)} = \tilde{t}^{(A,\phi)}. \]
Hence, $\tilde{t}$ is $G^{\vee}$-invariant
and the claim is proven for the conjugacy class of the unit. % element.

More generally, consider an element $h \in G$ in the center of $G$.
As in the proof of Proposition~\ref{rem:successive_quotients} the functor
%By the universal property of $\CC_G$ the functor
$\rho_h \colon \CC \to \CC$ lifts to the functor
\[ \tilde{\rho}_h \colon \CC_{G} \to \CC_{G},\ (A,\phi) \mapsto (\rho_h A, \rho_h \phi). \]
Let $t \colon \id_{\CC} \to \rho_h[i]$ be a $G$-invariant natural transformation.
As before, $t$ lifts to a natural transformation
$\tilde{t} \colon \id_{\CC_G} \to \tilde{\rho}_h[i]$
with $\tilde{t}^{(A,\phi)} = t^A$.

Consider the natural transformation $\Phi_h \colon \id_{\CC_G} \to \tilde{\rho}_h$ defined by
\[ (\Phi_h)^{(A,\phi)} = \phi_h \colon A \to h A. \]
Unlike before, this morphism depends on the linearization: For every character $\chi$
we have 
$(\Phi_h)^{(A,\chi \phi)} = \chi(h) \phi_h$,
therefore
\[ (\Phi_h)^{\chi (A,\phi)} = \chi(h) (\Phi_h)^{(A,\phi)}. \]

Consider now the composition
\[ \tilde{t} \tilde{\rho}_{h^{-1}} \circ \Phi_{h^{-1}} \colon \id_{\CC_G} \to \tilde{\rho}_{h^{-1}} \to \id_{\CC_G}[i] \]
%\[ \Phi_h^{-1}[i] \circ \tilde{t} \colon \id_{\CC_G} \to \tilde{\rho}_h[i] \to \id_{\CC_G}[i] \]
which under the isomorphism of Lemma~\ref{lemma124} is precisely the Hochschild cohomology class corresponding to $t$.
Then we have
\begin{align*}
\left( \chi (\tilde{t} \tilde{\rho}_{h^{-1}} \circ \Phi_{h^{-1}})\chi^{-1} \right)^{(A,\phi)}
& = \chi \left( \tilde{t} \tilde{\rho}_{h^{-1}} \circ \Phi_{h^{-1}} \right)^{\chi^{-1}(A,\phi)} \\
& = \left( \tilde{t} \tilde{\rho}_{h^{-1}} \circ \Phi_{h^{-1}} \right)^{\chi^{-1}(A,\phi)} \\
& = ( \tilde{t} \tilde{\rho}_{h^{-1}} )^{\chi^{-1} (A,\phi)} \circ \left( \Phi_{h^{-1}} \right)^{\chi^{-1}(A,\phi)} \\
& = ( \tilde{t} \tilde{\rho}_{h^{-1}} )^{(A,\phi)} \circ \chi^{-1}(h^{-1}) \cdot \left( \Phi_{h^{-1}} \right)^{(A,\phi)} \\
%(\Phi_h^{-1})^{(A,\phi)} \circ \tilde{t}^{(A,\phi)} \\
& = \chi(h) \cdot (\tilde{t} \tilde{\rho}_{h^{-1}} \circ \Phi_{h^{-1}})^{(A,\phi)}.
\end{align*} 
\begin{comment}
\begin{align*}
\left( \chi (\Phi_h^{-1} \circ \tilde{t})\chi^{-1} \right)^{(A,\phi)}
& = \chi \left( \Phi_h^{-1} \circ \tilde{t} \right)^{\chi^{-1}(A,\phi)} \\
& = \left( \Phi_h^{-1} \circ \tilde{t} \right)^{\chi^{-1}(A,\phi)} \\
& = ( \Phi_h^{-1} )^{\chi^{-1} (A,\phi)} \circ \tilde{t}^{\chi^{-1}(A,\phi)} \\
& = \chi(h) (\Phi_h^{-1})^{(A,\phi)} \circ \tilde{t}^{(A,\phi)} \\
& = \chi(h) (\Phi_h^{-1} \circ \tilde{t})^{(A,\phi)}. \qedhere
\end{align*} 
\end{comment}

We finally consider the general case given by a conjugacy class $\Fc$.
By tracing through the isomorphism of Lemma~\ref{lem:perry}
(compare also to the proof of Proposition~\ref{prop:hom_components_equiv})
one checks that a $G$-invariant natural transformation
\[ t = \bigoplus_{h \in \Fc} t_h \colon \id_{\CC} \to \bigoplus_{h \in \Fc} \rho_h \]
corresponds to the transformation\footnote{One checks that the left hand side
is $G$-invariant for any $(A,\phi) \in \CC_G$ and hence defines a well-defined natural transformation of $\id_{\CC_G}$.}
\[ \sum_{h \in \Fc} \tilde{t}_h \tilde{\rho}w_{h^{-1}} \circ \Phi_{h^{-1}} \colon \id_{\CC_G} \to \id_{\CC_G}. \]
The claim follows now by the same argument as before.
\end{proof}

\subsection{Hochschild homology and Serre functors}
From now on we assume that the category $\CC$ is equipped with a Serre functor $S$.
The Hochschild homology of $\CC$ is then defined by
\[ \HH_{\bullet}(\CC) \coloneqq \HH^{\bullet}(\CC, S) \]
%Assume that the category $\CC$ is equipped with a Serre functor $S$.
%The Hochschild homology of $\CC$ is defined by
%\[ \HH_{\bullet}(\CC) \coloneqq \HH^{\bullet}(\CC, S) \]
By Lemma~\ref{lemma123} (or, more precisely, its analogue for dg-categories) any equivalence $F \colon \CC \to \CC$ commutes with the Serre functor $S$ in a canonical way.
Hence, $F$ acts on Hochschild homology by conjugation\footnote{See also \cite[Sec.\ 5]{Perry} for the construction of this action in the dg-setting.}
\[ F_{\ast} \colon \HH_{\bullet}(\CC) \to \HH_{\bullet}(\CC), \ a \mapsto F a F^{-1}. \]
Moreover, arguing as in Section~\ref{subsec:Serrefunctor} shows that the Serre functor $S$
lifts to a Serre functor $\tilde{S}$ on the equivariant category $\CC_{G}$.

\begin{rmk}
The existence of Serre functors on the equivariant category can also be argued more abstractly using the notions of 'smooth', 'proper' and 'dualizable' dg-categories
for which we refer to Part I of \cite{Perry2}. This proceeds as follows. %, but may be skipped. %The reader not familiar with  may very well skip this section.

Let $\CC$ be a smooth and proper dg-category, for example the bounded derived dg-category of coherent sheaves on a smooth and proper variety.
Then $\CC$ is dualizable and hence admits a Serre functor $S$, see \cite[Sec.\ 4.6]{Perry2}.
If we have a finite group $G$ acting on $\CC$, then since the morphisms in the equivariant category $\CC_G$ are precisely the $G$-invariant morphisms in $\CC$,
it is immediate that $\CC_G$ is again proper.
We claim that $\CC_G$ is also smooth: Indeed, we need to show that $\id_{\Ind(\CC_G)} \in \Func( \Ind(\CC_G), \Ind(\CC_G) )$ is a compact object.
By \cite[Lem.\ 3.18]{BeckmannOberdieckModuli} or \cite[Lem.\ 3.6]{Perry} we have $\Ind(\CC_G) = \Ind(\CC)_G$ and combining with \cite[Lem.\ 4.7]{Perry} this yields an isomorphism
\[  \Func( \Ind(\CC_G), \Ind(\CC_G) ) \cong  \Func( \Ind(\CC), \Ind(\CC) )_{G \times G}. \]
By \cite[Lem.\ 3.7]{Perry} it suffices then to show that the image of $\id_{\Ind(\CC_G)}$ under the forgetful morphism
$\Func( \Ind(\CC), \Ind(\CC) )_{G \times G} \to \Func( \Ind(\CC), \Ind(\CC))$ is compact.
But according to \cite[Lem.\ 4.7]{Perry} this image is precisely $\oplus_{g \in G} \rho_g$ and hence as a finite direct sum of compact objects compact
(to show that $\rho_g$ is compact, use that $\id_{\Ind(\CC)}$ is compact, and since $\rho_g$ is invertible, left multiplication by it preserves compact objects,
therefore also $\rho_g \id_{\Ind(\CC)} = \rho_g$ is compact).

Having that $\CC_G$ is both smooth and proper, we find that it is dualizable and hence admits a Serre functor $\tilde{S}$.
%Since the left and right adjoint of the forgetful functor $p : \CC_G \to \CC$ agree one finds $p \tilde{S} = S p$.
Using that by Yoneda's Lemma every Serre functor is characterized by the defining isomorphism
$\Hom(\tau(-, -))^{\vee} \cong \Hom( - , S (- ))$, where $\tau$ is the functor that swaps th two factors (see \cite[Proof of Lem.\ 4.19]{Perry2}),
one finds that the Serre functor $\tilde{S}$ agrees with the one constructed in Section~\ref{subsec:Serrefunctor}.
%Assume that $\CC$ is a smooth 
\end{rmk}

\subsection{Calabi--Yau categories} 
Suppose now further that $\CC$ is a \emph{Calabi--Yau category}, i.e.\ that there exists a $2$-isomorphism
\[ a \colon \id_{\CC} \xrightarrow{\cong} S[-n] \]
for some integer $n$, called the dimension of $\CC$. In particular, we have
%In particular, if $\CC$ is Calabi--Yau, i.e.\ $S \cong \id_{\CC}[n]$ for some $n$ we have
\[ \HH_{\bullet}(\CC) \cong \HH^{\bullet}(\CC)[n]. \]
%Arguing as in Section~\ref{Subsubsection_Serre_functor} the Serre functor $S$ lifts to an Serre functor on the equivariant category.
We can ask under which conditions is the equivariant category $\CC_G$ again Calabi--Yau. The answer is very natural and given as follows.

\begin{lemma} \label{lemma_equi_cat_CY}
If the homology class $(a \colon \id_{\CC} \xrightarrow{\cong} S[-n]) \in \HH_{\bullet}(\CC)$ is invariant under the action of $G$, then $\CC_{G}$ is Calabi--Yau of dimension $n$.
Moreover, the induced action of $G^{\vee}$ on $\HH_{\bullet}(\CC_G)$ preserves the class
of the Calabi--Yau form.
\end{lemma}
\begin{proof}
As in the proof of Lemma~\ref{lemma125}, if $a$ is $G$-invariant, then the morphism
\[ \tilde{a}^{(A,\phi)} \colon (A,\phi) \to (SA[-n], S\phi[-n]) \]
given by $a^A\colon A \to SA[-n]$ defines a lift
\[ \tilde{a} \colon \id_{\CC_G} \to \tilde{S}[-n]. \]
Similarly, the inverse of $\tilde{a}$ also lifts. Hence, $\tilde{a}$ is a $2$-isomorphism.

Moreover, since $\tilde{a}$ is a lift of a $G$-invariant class,
arguing as in the proof of Lemma~\ref{lemma125} shows that
it is fixed by the action of $G^{\vee}$ by conjugaction.
\end{proof}

\subsection{Hochschild and singular cohomology}
\label{subsec:Hochschild-singular-coh}

Let $X$ be a smooth projective variety and let $D_{dg}(X)$ be the dg-enhancement of $D^b(X)$.
The Hochschild homology and cohomology of $X$ are defined by
\[ \HH_{\bullet}(X) = \HH_{\bullet}( D_{dg}(X) ), \quad \HH^{\bullet}(X) = \HH^{\bullet}( D_{dg}(X) ). \]
One has the Hochschild--Kostant--Rosenberg (HKR) isomorphism
\[ \HH_{n}( D(X) ) \cong \bigoplus_{p-q=n} H^q(X, \Omega_X^p). \]

Consider a Fourier--Mukai transform $\FM_{\CE} \colon D_{dg}(X) \to D_{dg}(Y)$
for a kernel $\CE \in D_{dg}(X \times Y)$.
If $\FM_{\CE}$ is an equivalence, this defines an isomorphism of Hochschild homology by conjugation
\begin{gather*} \FM_{\CE, \ast} \colon \HH_{\bullet}(X) \to \HH_{\bullet}(Y), \ a \mapsto \FM_{\CE} a \FM_{\CE}^{-1}.
% \\  (a \colon \id_{D_{dg}(X)} \to S[i]) \mapsto (\varphi a \varphi^{-1} \colon \id_{D_{dg}(Y)} \to \varphi S \varphi^{-1}[i] = S[i]). 
\end{gather*}

The kernel $\CE$ also induces a morphism on singular cohomology
\[ \FM_{\CE,\ast} \colon H^{\ast}(X) \to H^{\ast}(Y), \alpha \mapsto q_{\ast}( p^{\ast}(\alpha) \cdot v(\CE) ) \]
where $p,q$ are the projections of $X \times Y$ onto the factors and we let $v(E) = \ch(E) \sqrt{\td_X}$ denote the Mukai vector.
By the main result of \cite{CRV} the action of $\FM_{\CE}$ on Hochschild and singular cohomology are compatible under the HKR isomorphism, i.e.\ the following diagram commutes:
\[
\begin{tikzcd}
\HH_{\bullet}(X)  \ar{d}{\textup{HKR}} \ar{r}{\FM_{\CE,\ast}} &  \HH_{\bullet}(Y) \ar{d}{\textup{HKR}} \\
H^{\ast}(X,\BC) \ar{r}{\FM_{\CE, \ast}} & H^{\ast}(Y, \BC).
\end{tikzcd}
\]

In particular, in order to apply Lemma~\ref{lemma_equi_cat_CY}
to $D_{dg}(X)$ it suffices to check the invariance of the Calabi--Yau form
on singular cohomology, i.e.\ that the element in
\[ H^0(X,\omega_X) = H^{n,0}(X,\BC) \]
corresponding to the isomorphism $\omega_X \cong \CO_X$ is preserved by $G$.

\section{Equivariant categories of elliptic curves}
\label{sec:equiv_cat_ell_curves}
We illustrate some of the results of the previous sections
by applying them to the example of the derived category of coherent sheaves on an elliptic curve.

Let $E$ be a non-singular elliptic curve and let
$G$ be a finite group which acts on $D^b(E)$.
By Lemma~\ref{lemma:FMaction} the action is given by Fourier--Mukai transforms
and hence induces an action on cohomology.
We assume that each $g\in G$ fixes the Calabi--Yau form, i.e.\ that the cohomological Fourier--Mukai transform $\FM_{\CE_g,\ast}$ acts trivially on $H^{1,0}(E)$. 
In this case we also say that the $G$-action on $D^b(E)$ is \emph{Calabi--Yau}.

Our goal in this section is to prove the following:
\begin{thm} \label{thm:elliptic curve}
For any Calabi--Yau action of $G$ on $D^b(E)$,
the equivariant category $D^b(E)_{G}$ decomposes into finitely many derived categories of elliptic curves.
\end{thm}

The subgroup of $\Aut D^b(E)$ acting trivially on cohomology is isomorphic to
\begin{equation} \label{ZxExPic}
\BZ \times E \times \Pic^0(E)
\end{equation}
where the first summand is identified with the shift by $[2]$, the second one with the pullback along translations and the last one with tensoring with degree~0 line bundles. 
We first show that any Calabi--Yau action acts trivially on cohomology.

\begin{lemma}
For any Calabi--Yau action by a finite group $G$ on $D^b(E)$
the induced action on $H^\ast(E,\BZ)$ is trivial.
%Consider a Calabi--Yau action by $G$ on $D^b(E)$.
%Then for any $g \in G$ the cohomological Fourier--Mukai transform $\FM_{\CE_g,\ast}$ acts trivially on the integral cohomology $H^\ast(E,\BZ)$. 
\end{lemma}
\begin{proof}
Recall that auto-equivalences of elliptic curves induce Hodge isometries on the integral cohomology lattice. 
Since $H^1(E,\BZ)$ is preserved by $\FM_{\CE_g,\ast}$ and by assumption the action on $H^{1,0}$ is trivial, it follows that $\FM_{\CE_g,\ast}$ acts trivially on $H^1(E,\BZ)$. 

The isometry $\FM_{\CE_g,\ast}$ induces an isometry of $H^0(E,\BZ) \oplus H^2(E,\BZ)$ which, as a lattice, is isometric to the hyperbolic plane $U$. We have $\Aut U = \BZ_2 \times \BZ_2$. Moreover, only $\pm \id$ can occur in our case, since the action of every auto-equivalence is orientation-preserving. Let us denote this isometry by $\tau$. 

Let $\iota \colon E \to E$ denote the morphism given by multiplication with $-1$.
Observe that the auto-equivalence $F=[1] \circ \iota$ fixes $H^1(E,\BZ)$ and acts as $-\id$ on $H^0(E)\oplus H^2(E)$, hence its cohomological Fourier--Mukai transform yields $\tau$. 
Given any auto-equivalence $F' \colon D^b(E) \to D^b(E)$ which induces $\tau$,
the composition $F \circ F'$ acts trivial on cohomology and hence lies in the subgroup \eqref{ZxExPic}.
It follows that
%Since $G$ sends coherent sheaves to coherent sheaves shifted by an odd degree,
every auto-equivalence inducing $\tau$ on cohomology has infinite order.
Hence, it can never be contained in the image of $G$ in $\Aut \CD$.
\end{proof}

\begin{proof}[Proof of Theorem~\ref{thm:elliptic curve}]
From the lemma it follows that the image of $G$ in $\Aut D^b(E)$ maps to $E \times \Pic^0(E)$.
This shows that $G$ preserves the category $\Coh(E)$ 
and hence lifts to an action of a dg-enhancement of $D^b(E)$.
By Theorem~\ref{prop:faithful action} we may assume that the action is faithful.
Moreover, by Lemma~\ref{lemma_equi_cat_CY} the equivariant derived category 
\[ D^b(E)_G \cong D^b(\Coh(E)_G) \]
is again a 1-Calabi--Yau category. 
In particular, $\Coh(E)_G$ is an indecomposable 1-Calabi--Yau abelian category. 

To prove the claim we apply the classification \cite[Thm.\ 1.1]{VanRoosmalen} of such categories. More precisely, the proof of \cite[Thm.\ 4.7]{VanRoosmalen} shows that as soon as $\Coh(E)_G$ has two non-isomorphic simple objects (i.e.\ objects $T$ such that $\Hom(T,T)=\BC$), then $\Coh(E)_G$ is equivalent to the category $\Coh(E')$ for some elliptic curve $E'$.

Hence we need to find two simple, non-isomorphic objects in $\Coh(E)_G$. Since $G'$ acts trivially on the stability manifold, we know from Section~\ref{subsec:stab_conditions} that $D^b(\Coh(E)_G)$ has again a stability condition satisfying the support property. In particular, finite Jordan--Hölder filtrations of semistable objects exist. If, up to isomorphism, there would only exist one simple object, then there would be only one stable object $T$ in $D^b(\Coh(E)_G)$. 
Since $pq$ acts by multiplication by $|G|$ on the numerical $K$-group of $\Coh(E)$,
this implies that $p(E)$ generates up to finite index the numerical $K$-group of $E$.
However, clearly we have 
\[ K_{\textup{num}}(E)=\BZ^2 \]
which gives a contradiction.
\end{proof}

\end{document}